\documentclass[a4paper,12pt]{amsart}
\usepackage[utf8]{inputenc}
\usepackage[T1]{fontenc}
\usepackage[UKenglish]{babel}
\usepackage[margin=18mm]{geometry}

\usepackage{color}%cyan,red;magenta,green;yellow,blue

\usepackage{graphicx}

\usepackage{amsmath,amssymb,amsfonts,amsthm}
\usepackage{mathrsfs,eucal,dsfont}
\usepackage{verbatim,enumitem}

\usepackage{hyperref,url}

\newcommand{\R}{\mathds R}

\newcommand{\Pp}{\mathds P}
\newcommand{\Ee}{\mathds E}

\newcommand{\I}{\mathds 1}

\def\aa{\alpha}

\def\d{{\rm d}}
\def\<{\langle}
\def\>{\rangle}

 \def\ss{\sqrt}

\def\R{\mathbb R}   \def\ss{\sqrt} 
 \def\kk{\kappa} 
  \def\vv{\varepsilon} 
\def\<{\langle} \def\>{\rangle}  
  \def\nn{\nabla}  
\def\d{\text{\rm{d}}}  \def\aa{\alpha} 
  \def\si{\sigma} 
 \def\beq{\begin{equation}}  
 
\def\e{\text{\rm{e}}}  \def\OO{\Omega}  
  
 \def\P{\mathbb P}

\def\to{\rightarrow}
\def\8{\infty}\def\3{\triangle}
\def\1{\lesssim}

\renewcommand{\bar}{\overline}
\renewcommand{\hat}{\widehat}

\newtheorem{theorem}{Theorem}[section]

\newtheorem{proposition}[theorem]{Proposition}

\theoremstyle{definition}

\newtheorem{example}[theorem]{Example}
\newtheorem{remark}[theorem]{Remark}

\numberwithin{equation}{section}
\begin{document}
\allowdisplaybreaks

\title[L\'{e}vy driven Langevin dynamics with singular potentials] {Exponential ergodicity of L\'{e}vy driven Langevin dynamics with singular potentials}

\author{
Jianhai Bao\qquad Rongjuan Fang\qquad
Jian Wang}
\date{}
\thanks{\emph{J.\ Bao:} Center for Applied Mathematics, Tianjin University, 300072  Tianjin, P.R. China. \url{jianhaibao@tju.edu.cn}}
\thanks{\emph{R.\ Fang:}
School  of Mathematics and Statistics, Fujian Normal University, 350007 Fuzhou, P.R. China. \url{fangrj@fjnu.edu.cn}}

\thanks{\emph{J.\ Wang:}
School  of Mathematics and Statistics \&  Fujian Key Laboratory of Mathematical
Analysis and Applications (FJKLMAA) \&  Center for Applied Mathematics of Fujian Province (FJNU), Fujian Normal University, 350007 Fuzhou, P.R. China. \url{jianwang@fjnu.edu.cn}}

\maketitle

\begin{abstract} In this paper, we address exponential ergodicity for L\'{e}vy driven Langevin dynamics with singular potentials, which can be used to model
the time evolution of a molecular system
consisting of $N$ particles moving in $\R^d$
and subject to discontinuous stochastic forces. In particular, our results  are applicable to the singular setups concerned with not only the Lennard-Jones-like interaction potentials but also the Coulomb potentials. In addition to Harris' theorem,
the approach is based on novel constructions of proper Lyapunov functions (which are completely different from the setting for Langevin dynamics driven by Brownian motions), on invoking    the H\"{o}rmander theorem for non-local operators and on solving the issue on an approximate controllability of the associated deterministic system  as well as on exploiting  the time-change idea.

\medskip

\noindent\textbf{Keywords:} Langevin dynamic; L\'evy noise; singular potential; exponential ergodicity; Lyapunov function

\smallskip

\noindent \textbf{MSC 2020:}  60K35,  37A25, 60J76
\end{abstract}
\section{Introduction and Main Result}
In physics, the Langevin dynamics is used to model
the time evolution of a molecular system
consisting of $N$ particles moving in $\R^d$, and it can be described mathematically by the following degenerate SDE
on $\R^{Nd}\times \R^{Nd}:=(\R^d)^N\times (\R^d)^N$:
\begin{equation}\label{E0}
\begin{cases}
\d   {\bf X}_t=  \nn _{{\bf V}}H( {\bf X}_t,{\bf V}_t) \,\d t,\\
\d {\bf V}_t=- \big(F( {\bf X}_t,{\bf V}_t)\nn _{{\bf V}}H( {\bf X}_t,{\bf V}_t)+\nn _{{\bf X}}H( {\bf X}_t,{\bf V}_t)\big)\,\d t+\d  {\bf Z}_t,
\end{cases}
\end{equation}
where ${\bf X}_t:=\big(X_t^{(1)},X_t^{(2)},\cdots, X_t^{(N)}\big)\in  \R^{Nd} $ and $ {\bf V}_t:=\big(V_t^{(1)},V_t^{(2)},\cdots, V_t^{(N)}\big)\in \R^{Nd} $
 represent the positions   and the velocities of $N$ particles, respectively;  $H:\R^{Nd}\times \R^{Nd}\to\R$ means the Hamiltonian function; $F:\R^{Nd}\times \R^{Nd}\to[0,\8]$ refers to the damping coefficient; and $({\bf Z}_t)_{t\ge0}$ is an $\R^{Nd}$-valued noise process. For  detailed  physical backgrounds as well as  more applications in mechanics, the readers are referred to the monograph \cite{Soize}.
In particular, when the Hamiltonian energy $H(x,v)=\frac{1}{2}|v|^2+U(x)$ for a smooth $U:\R^{Nd}\to\R $ and the damping coefficient $F$ is constant (i.e., $F(x,v)=\gamma$ for some positive $\gamma$), \eqref{E0} reduces to
\begin{equation}\label{E1}
\begin{cases}
\d   {\bf X}_t=  {\bf V}_t \,\d t,\\
\d {\bf V}_t=-\big(\gamma   {\bf V}_t +\nn U({\bf X}_t)\big)\,\d t+\d  {\bf Z}_t.
\end{cases}
\end{equation}
Herein,
 $-\gamma {\bf V}_t$ stands for the damping force with the magnitude $\gamma>0$ of the friction arising from the thermal medium; $\nn$ is the gradient operator on $\R^{Nd}$   and
  $U: \R^{Nd} \to [0,\8]$ denotes the potential energy,
which might incorporate the confining potential due to external forces and the interaction potential via repulsive forces.

In the past few decades, the long time behavior (for example, the exponential  ergodicity)  of \eqref{E1} with a single particle (i.e., $N=1$)
has been developed considerably. In case that
 the potential $U$ is polynomial-like  and the driven noise $({\bf Z}_t)_{t\ge0}$ is a $d$-dimensional Brownian motion, the geometric ergodicity under the total variation distance  was treated in, for instance, \cite{MSH,Talay,Wu} via Harris' theorem (e.g. \cite[Theorem 1.2]{HMa}), and
  \cite{EGZ} with the aid of a reflection coupling approach. Besides the previous two probabilistic methods, an important analytical tool on investigating exponential convergence of \eqref{E1} with a $C^2$-potential $U$ to equilibrium in $H^1(\mu)$ or $H^2(\mu)$ is
the hypocoercivity theory  initiated by Villani \cite{Villani}.

From the point of view   on  statistical mechanics, the interaction potentials, characterizing the repulsive forces,  exhibit certain singular features since the interactions increase dramatically   when particles approach each other. Among the singular  interaction potentials,
the Lennard-Jones potential and the Coulomb potential are two typical candidates. In contrast to the setting that the potential term $U$ is regular,
it is much more challengeable to investigate the long time behavior of \eqref{E1}
 due to the involvement of repulsive potentials. All the same, great progresses have been made concerning the exponential ergodicity of \eqref{E1} when the underlying noise $({\bf Z}_t)_{t\ge0}$ is a standard Brownian motion. Based on a delicate construction of Lyapunov function and a perfect  application of Harris' theorem, exponential ergodicity under the weighted total variation distance was explored in \cite{HM} for admissible potentials, which include the Lennard-Jones-like interaction potential but exclude the Coulomb potential. In \cite{LM}, the authors go further and construct another novel Lyapunov function to tackle  the geometric ergodicity of Langevin dynamics with Coulomb interactions by examining the criteria on Harris' theorem.
Furthermore, in case that  the repulsive potential satisfies a global integrability condition,  the exponential ergodicity under the weighted total variation distance was considered in \cite{SX} by taking advantage of Zvonkin's transform  to remove the singular potential, where not only the Lennard-Jones potential but also the Coulomb potential are exclusive totally.
In the meantime, $L^2$-exponential  ergodicity  of \eqref{E1} with singular potentials has also received much attention; see, for example,
\cite{CHS,GS} for more details.
Additionally, by developing the trick on the construction of  Lyapunov functions in  \cite{HM,LM}, the existence and the uniqueness of quasi-stationary distribution
 for hypoelliptic Hamiltonian dynamics with singular potentials were  addressed in \cite{GNW}.

For the past twenty years,
 there are increasing attentions paid into
 regularity properties and ergodic properties of fundamental solutions for kinetic Fokker-Planck operator associated with
  Langevin dynamics \eqref{E0} driven by L\'evy noises $({\bf Z}_t)_{t\ge0}$ for the case  $N=1$. In particular,
a series of papers (e.g., \cite{Zhang4, Zhang6, Zhang3, Zhang2, Zhang1, Zhang, Zhang5}) due to  Zhang and his coauthors  are devoted to H\"{o}rmander's theorem for non-local operators,
which are closely linked  with hypoellipticity theory of fractional kinetic equations or (linear models) of the spatially inhomogeneous Boltzmann equations without an angular cutoff; see e.g. \cite{A,Chen,MX}.
Furthermore, yet under the assumption that the potential $U$ is regular nonetheless the underlying noise $({\bf Z}_t)_{t\ge0}$ is a $d$-dimensional pure jump L\'evy process,
the exponential ergodicity under a multiplicative type quasi-Wasserstein distance of the Markov process solved  by  \eqref{E0}  was treated   in \cite{BW} by invoking a refined basic coupling method. Actually, motivated by sampling from the
heavy-tailed distribution arising in various applications including statistical machine learning \cite{SZT} and  statistical physics study \cite{CD},
the long time behaviors of the SDE \eqref{E1} with
 an $\alpha$-stable L\'evy motion
(instead of a Brownian
motion) will  play more appropriate and important roles.

As mentioned above, there is a huge amount of literature concerned with the long term behavior of \eqref{E1} when the driven noise is a standard Brownian motion.  Nevertheless, in some occasions, the stochastic system under consideration might be subject to a discontinuous stochastic force
rather than a continuous version. So,
strongly inspired by the motivations above, it should be
indispensable and
interesting to investigate the ergodicity of the SDE \eqref{E1} driven by L\'evy noises in the $N$-particles framework with $N>1$, in particular, when the potential term $U(x)$ is singular due to the interactions among the particles. So far, the research in this aspect is still vacant, and therefore we intend to proceed to close the corresponding gap.

In detail, throughout this paper we always assume that $  {\bf Z}_t:=(Z_t^{(1)},\cdots,Z_t^{(N)})\in \R^{Nd} , $
where $ (Z_t^{(1)})_{t\ge0}$, $\cdots,$ $(Z_t^{(N)})_{t\ge0}$ are mutually independent
$d$-dimensional    pure jump L\'{e}vy processes, defined on a common  probability space $(\OO,\mathscr F,\Pp)$,  with the L\'{e}vy measures $\nu^{(1)}(\d u),\cdots, \nu^{(N)}(\d u)$, respectively. Let $P_t((x,v),\cdot)$ be the transition kernel of the process  $({\bf X}_t,{\bf V}_t)_{t\ge0}$ determined  by \eqref{E1}.

\ \

Below, to avoid  writing   down intricate assumptions on the potential term $U$ and, most importantly of all,  state succinctly the contribution of this paper,
we focus on the settings related to  the Lennard-Jones-like potential and the Coulomb potential (as two very typical representatives of singular potentials),  rather than much more general potentials,
 to present  the main result of the present paper.

\begin{theorem}\label{thm1} Let the driven noise $ ( {\bf Z}_t)_{t\ge0}:=((Z_t^{(1)},\cdots,Z_t^{(N)}))_{t\ge0}$ be so that, for any $1\le i\le N$,
$(Z^{(i)}_t)_{t\ge0}$ is a $d$-dimensional $($rotationally invariant$)$  symmetric $\alpha_i$-stable L\'{e}vy process with $\alpha_i\in (0,2)$, and $ (Z_t^{(1)})_{t\ge0}$, $\cdots,$ $(Z_t^{(N)})_{t\ge0}$ are mutually independent.
Suppose one of the following conditions holds:
\begin{itemize}
\item[{\rm (i)}] $(${\bf the Lennard-Jones-like potential}\,$)$ there are constants $c_0,c_1,c_2>0$ so that for all $x\in \R^{Nd}$,
$$
U(x)=c_0\sum_{i=1}^N (1+|x^{(i)}|^2)^{\alpha/2} +\sum_{1\le i<j\le N}U_I(x^{(i)}-x^{(j)}),
$$ where  $\alpha\ge2$, and for a non-zero $u\in\R^d,$ $$U_I(u):=c_1|u|^{-12}-c_2 |u|^{-6}.$$
\item [{\rm (i)}]  $(${\bf the Coulomb potential}\,$)$ there is a constant $c_0>0$ so that for all $x\in \R^{Nd}$,
$$
U(x)=c_0\sum_{i=1}^N (1+|x^{(i)}|^2)^{\alpha/2} +\sum_{1\le i<j\le N}U_I(x^{(i)}-x^{(j)}),
$$ where $\alpha\ge2$, and for a non-zero $u\in\R^d,$
  $$U_I(u):=|u|^{2-d},\quad d\ge 3;\,\, U_I(u):=-\log |u|,\quad d=2.$$\end{itemize}
Then, the process $({\bf X}_t,{\bf V}_t)_{t\ge0}$ given by \eqref{E1} is exponentially ergodic in the sense that the process $({\bf X}_t,{\bf V}_t)_{t\ge0}$ has a  unique invariant probability measure $\mu$, and that there are a constant $\lambda$ and a positive function $C(x,v)$ so that for all $(x,v)\in \mathscr{K}:=\mathscr D(U)\times\R^{Nd}$ and $t>0$,
$$\|P_t((x,v),\cdot)-\mu\|_{V}\le C(x,v)\e^{-\lambda t},$$ where, for a signed measure $\mu_0$,
$ \|\mu_0\|_V=\sup_{|f|\le V}|\mu_0 (f)|$, and $V(x,v)\ge 1$ for all $(x,v)\in \mathscr{K}$ satisfying that
$$V(x,v)\simeq \left(|v|^2+U(x)\right)^{\theta/2}\quad \hbox{ as   } |v|^2+U(x) \to \infty$$
with $\theta\in (0, \min_{1\le i\le N}\alpha_i).$
\end{theorem}

Here and in what follows, for two functions $f,g,$ $f\simeq g$ means that there are constants $c_*,c^*$ such that $c_*f\le g\le c^*f$. In the following, we make some comments on Theorem \ref{thm1} and its proof.
\begin{remark}
\begin{itemize}

\item[{\rm(i)}] To be sure, the main result above is described in two specific setups.  Concerning    general frameworks, which are  applicable to much more wider singular interaction potentials, the associated main result  on exponential ergodicity of \eqref{E1} under the weighted total variation distance
will be
given  in Section \ref{section4}.

\item[{\rm(ii)}]
 Harris' theorem  is one of powerful tools to  investigate exponentially ergodic properties of Markov processes, where one of the essentials
is to construct an appropriate Lyapunov functions.
In \cite{HM,LM}
(even in \cite{MSH, Wu} for the case that the coefficients are regular), the exponential type Lyapunov functions (see (2.4) in \cite{HM} and (16) in \cite{LM} for more details) were constructed
in order to investigate
the geometric ergodicity of Langevin dynamics with singular potentials.
Such kind Lyapunov functions in turn require that the associated process has finite exponential type moments; see \cite[Theorem 2.3]{HM} and \cite[Theorem 2.5]{LM} for more details.
However, this requirement is rather restrictive  in case that  the driven noise (e.g., an $\alpha$-stable L\'{e}vy motion) admits heavy-tail properties. In particular, regarding the framework we are interested in, the Lyapunov functions constructed in
\cite{HM,LM} are not available any more. Yet motivated by inspirations on construction of Lyapunov functions in \cite{HM,LM}, two completely novel Lyapunov functions are built in the present work and applicable to singular potentials (e.g., the Lennard-Jones-like potential and the Coulomb potential).
In contrast to the main results in \cite{HM,LM}, the Lyapunov functions established in Theorem \ref{thm1} are only allowed to be of polynomial growth.
More importantly, the Lyapunov function given the present work also pave the way for further investigating weighted $L^2$-contractivity (e.g. \cite{CHS}), and the existence and the uniqueness of quasi-stationary distribution (e.g. \cite{GNW}) of L\'{e}vy driven Langevin dynamics with singular potentials.

\item[{\rm(iii)}]
With regarding to the construction of Lyapunov functions,
the driven L\'evy noises involved
can be much more general.  Whereas, as for  the irreducible property, the noise term is confined to be a range of symmetric stable processes (or more general subordinated Brownian motions), where the key ingredient is to make fully use of the topologically irreducible property of Brownian motions. To evade the application of Harris' theorem (in the vast majority of occasions, it is extremely difficult to  examine the strong Feller property and the irreducibility), the probabilistic coupling method might be an alternative to deal with ergodicity of \eqref{E1}
and, most importantly,
to encompass a wide range of pure jump L\'{e}vy noises
 (see \cite{BW} for the case that the coefficients are regular). Due to the involvement of singular potential, for the moment, it is still a very formidable task to construct an appropriate coupling and a suitable metric function to explore the ergodicity under the (quasi-)Wasserstein distance. Although it is arduous, it is worthy to make an attempt in the forthcoming  work.

\end{itemize}

\end{remark}

The rest of this paper is arranged as follows. In  Section \ref{sec2}, we aim to construct appropriate Lyapunov functions (see Propositions \ref{pro1} and \ref{pro2}) for the SDE \eqref{E1} under two different sufficient conditions. In particular, these two settings work very well for the Lennard-Jones-like interaction potentials (see Example \ref{ex1}) and the Coulomb potentials (see Example \ref{ex2}), respectively.  Section \ref{section3} is devoted to the strong Feller property (see Proposition \ref{strong}) by
adopting
a truncation argument and employing the
 H\"ormander theorem for nonlocal operators,
and to the irreducible property (see Proposition \ref{irre}) via the trick on the approximate controllability of the associated deterministic system,  where the driven noise $ ( {\bf Z}_t)_{t\ge0}:=((Z_t^{(1)},\cdots,Z_t^{(N)}))_{t\ge0}$ involved in \eqref{E1} is an independent
symmetric stable L\'{e}vy process. Meanwhile, with the Lyapunov function, the strong Feller property and the irreducibility at hand, the proof of Theorem \ref{thm1} is
 given in Section \ref{section4}
 by invoking Harris' theorem.
 Finally, a much more general result on exponential ergodicity of \eqref{E1} is presented
  before the conclusion of this work.

\section{Lyapunov Functions and Lyapunov Condition}\label{sec2}
Let $\mathscr L$ be the infinitesimal generator of $({\bf X}_t,{\bf V}_t)_{t\ge0}$ determined by \eqref{E1}; that is, for all $\psi\in C^2_b(\mathscr{K})$,
\begin{equation}\label{E9}
\begin{split}
(\mathscr L\psi)(x,v)&=\<\nn_x\psi(x,v),v\>-\<\nn_v\psi(x,v),\gamma  v+\nn U( x)\>\\
&\quad+\sum_{i=1}^N\int_{\R^d}\big(\psi(x,v+ S_i( z))-\psi(x,v)-\<\nn_v^{(i)}\psi(x,v),z\>\I_{\{|z|\le 1\}}\big)\,\nu^{(i)}(\d z).
\end{split}
\end{equation}
Here and in what follows,  $\mathscr K:=\mathscr D(U)\times \R^{Nd}$ with $\mathscr D(U):=\{x\in \R^{Nd}: U(x)<\infty\}$; for  $v=(v^{(1)},\cdots,v^{(N)})\in\R^{Nd}$ with $v^{(i)}\in\R^d$,
$$
  \nn_v \psi(x,v)=(\nn^{(1)}_v\psi(x,v),\cdots, \nn^{(N)}_v\psi(x,v))\in\R^{Nd},
$$
 where $\nn^{(i)}_v$ means the gradient operator with respect to the component $v^{(i)}\in\R^d$; for $x\in\mathscr D(U)$, $ \nn_x \psi(x,v)$ is defined analogously;
for $z\in \R^d$, $S_i( z)$ is  the $i$-th substitution  of the zero vector $({\bf 0}_1,\cdots,{\bf 0}_N)\in\R^{Nd}$, i.e.,
$$  S_i( z) =({\bf0}_1,\cdots, {\bf 0}_{i-1}, z, {\bf0}_{i+1},\cdots, {\bf 0}_N).$$

The main purpose of this section is to provide  respectively two sufficient conditions for the existence of   Lyapunov functions for the process $({\bf X}_t,{\bf V}_t)_{t\ge0}$ solving \eqref{E1}. In detail, we will deal with  two cases separately, which are, in particular,  adaptable to  the Lennard-Jones type potential and the Coulomb type potential, respectively.

\subsection{Case 1}

In this part, we assume that the potential term $U$ and the L\'{e}vy measures $(\nu^{(k)})_{1\le k\le N}$ satisfy
\begin{enumerate}\it
\item[$({\bf H}_U)$] $U: \mathscr D(U) \to \R
$ is a $C^\infty$-function so that $U(x)$ is bounded from below such that $U(x)\to +\8$ if and only if  $x\to\partial \mathscr D(U) $
  $($the boundary of the domain $\mathscr D(U)$, which is open and path-connected$\,)$
or $|x|\to+\8$,
and that there is a constant $C_U>0$ such that for all $x\in \mathscr{D}(U)$,
\begin{equation}\label{E3-} U(x) (1+\| \nabla^2 U(x)\|)\le C_U(1+|\nabla U(x)|^2),\end{equation}
where $\nn^2$ and
$\|\cdot\|$ stand  for the Hessian operator and the Hilbert-Schmidt  norm, respectively.

\item[$({\bf H}_\nu)$] there exists a constant  $\theta\in(0,2]$  such that $$\sum_{i=1}^N\int_{\{|z|>1\}}|z|^{\theta}\,\nu^{(i)}(\d z)<\8.$$
\end{enumerate}

According to \eqref{E3-} and the fact that $U(x)$ is bounded from below so that  $U(x)\to +\8$ if and only if  $x\to\partial \mathscr D(U)  $ or $|x|\to+\8$, we have
\begin{itemize}
\item  there exists a constant $r_0>0$ such that $|\nn U(x)|\ge 1$ for any $x\in\mathscr D(U)$ with $U(x)\ge r_0$;

\item  for any $r>0,$ there exists a constant $C_{U,r}^{*}>0$ such that
\begin{equation}\label{E3*}
\sup_{x\in\mathscr D(U): U(x)\le r }\|\nn^2 U(x)\|\le C_{U,r}^{*};
\end{equation}

\item there exist constants $C^*_U,R_U^*>0$ such that for all $x\in\mathscr D(U)$ with $U(x)\ge R_U^*$,
\begin{equation}\label{E3}
 \frac{U(x)}{|\nn U(x)|^2} \big(1+\|\nn^2U(x)\|\big)\le C_U^*.
\end{equation}
\end{itemize}

Below, let $\alpha\in C^\8([0,\8);[0,1])$ satisfying
$$
\alpha(u)=
\begin{cases}
1,\quad u\ge 2r_0 ,\\
0,\quad u\le r_0,
\end{cases}
$$ and $|\alpha'(u)|\le {2}/{r_0}$ for all $u\in [r_0,2r_0]$.
For $\kk>0,$
define
$$\psi_\kk(x,v) =\frac{\kk}{|\nn U(x)|^2}\alpha(U(x)) U(x)\<v,\nn U(x)\>, \quad (x,v)\in\mathscr K.$$

According to the definition of $\alpha(\cdot)$ and the fact that $|\nn U(x)|\ge 1$ for all $x\in \mathscr{D}(U)$ with $U(x)\ge r_0$,
it follows  from \eqref{E3} that
\begin{equation}\label{E2}
\begin{split}
  |\psi_\kk(x,v)|&\le \frac{\kk^2 \alpha(U(x))^2U(x)^2}{|\nn U(x)|^2}+\frac{|v|^2}{4}\\ &\le \frac{\kk^2  U(x)^2}{|\nn U(x)|^2}\I_{\{U(x)\ge  r_0 \vee R_U^*\}}+\frac{\kk^2 U(x)^2}{|\nn U(x)|^2}\I_{\{r_0<U(x)<r_0 \vee R_U^*\}}+\frac{|v|^2}{4 }\\ &\le   \kk^2\Big( C^*_U \I_{\{U(x)\ge  r_0 \vee R_U^*\}} + U(x) \I_{\{r_0\le U(x)<  r_0  \vee R_U^* \}}\Big) U(x) +\frac{|v|^2}{4 }\\
  &\le \kk^2 \theta_0 U(x)\I_{\{U(x)\ge r_0\}}+\frac{|v|^2}{4 },\quad (x,v)\in\mathscr K,
\end{split}
\end{equation}
where $\theta_0:=r_0\vee C^*_U  \vee R_U^*  $.

In the following, we set
\begin{equation}\label{EE0}
 V_{\kk}( x, v):=C_*+ \frac{| v|^2}{2}+U( x)+\psi_\kk(x,v),\quad (x,v)\in\mathscr K,\end{equation} where
$$C_*:=1+ \sup_{x\in \mathscr{D}(U):U(x)< r_0} |U(x)|.$$
In particular, it holds  that for all $(x,v)\in\mathscr K$,
\begin{equation}\label{E4}
\begin{split}
C_*+\frac{3|v|^2}{4}+\big(1+\kk^2\theta_0\I_{\{U(x)\ge r_0\}}\big)U(x)\ge &V_{\kk}( x, v)\\
\ge &C_*+\frac{|v|^2}{4}+\big(1-\kk^2\theta_0\I_{\{U(x)\ge r_0\}}\big)U(x).
\end{split}
\end{equation}
Consequently,
as long as  $\kk\in\big(0,\frac{1}{\ss{\theta_0}}\big]$,  $V_{\kk}( x, v)\ge1$ for all $(x,v)\in\mathscr K$.

Furthermore, let
\begin{equation}\label{EE2*}
 \kk^* =\frac{1}{\ss{\theta_0}}\wedge \frac{1}{2\theta_0\gamma}\wedge\frac{\gamma}{4(5 +3  (C_U^*  \vee   (
 \beta_0  C^*_{U,\beta_0})))},
 \end{equation}
where $\beta_0:=(2r_0)\vee R_U^*$.

With the preliminary materials above,
we have the following extremely significant statement.

\begin{proposition}\label{pro1}
Assume that $({\bf H}_U)$ and $({\bf H}_\nu)$ hold. Then,
for any $\kk\in(0,\kk^*)$ and $\mathcal V_{\kk,\theta}(x,v):=V_\kk(x,v)^{{\theta}/{2}}$ with $\theta\in (0,2]$ given in $({\bf H}_\nu)$,  and $ \kk^*$ and
$V_{\kk}$ being defined by \eqref{EE2*} and \eqref{EE0} respectively, there exist constants $\lambda_{\mathcal V},C_{\mathcal V}>0$ such that for $(x,v)\in\mathscr K,$
\begin{equation}\label{E6}
(\mathscr L\mathcal V_{\kk,\theta})(x,v)\le-\lambda_{\mathcal V} \mathcal V_{\kk,\theta}(x,v)+C_{\mathcal V}.
\end{equation}
\end{proposition}

Before the proof of the proposition above, let us make some comments on the assumptions and the construction of the Lyapunov function $\mathcal V_{\kk,\theta}(x,v)$ involved in.

\begin{remark}
\begin{itemize}
\item[{\rm (i)}] Since $U(x)$ is bounded from below so that  $U(x)\to +\8$ if and only if  $x\to\partial \mathscr D(U)  $ or $|x|\to+\8$,
 it follows from \eqref{E3-} and \eqref{E3} that when $U(x)\to \infty$,
 $$|\nabla U(x)|\to \infty,\quad \frac{ \|\nabla^2 U(x)\|}{|\nabla U(x)|^2}\to 0.$$ Hence, the assumption $({\bf H}_U)$ is a little bit stronger than the properties of the admissible potential adopted in \cite{HM}; see \cite[Definition 2.1]{HM} for more details.
 Although the assumption $({\bf H}_U)$ is slightly restrictive compared with the counterpart in \cite{HM}, it is very competent to handle the singular potential we are interested in.

\item[{\rm (ii)}]  The construction of the Lyapunov function $\mathcal V_{\kk,\theta}(x,v):=V_\kk(x,v)^{{\theta}/{2}}$ has the following intuition from two aspects of viewpoints. Firstly, the function $V_\kk(x,v)$ given by \eqref{EE0} somehow is inspired by the exponent term of the function $W(q,p)$ given in \cite[(5.1)]{HM}. However, the lower order perturbation term $\psi_\kk(x,v)$ here is different entirely  from that in \cite{HM}. Indeed, this difference is crucial to our arguments for the non-local operator $\mathscr L$.  Secondly, since we assume that the L\'evy measure has only finite $\theta$-moment condition as stated in Assumption $({\bf H}_\nu)$, the exponential-type Lyapunov function $W(q,p)$ used in \cite{HM} does not work in our setting. Instead, we will take the power-order (exactly with the $\theta/2$-order under Assumption $({\bf H}_\nu)$) of the function $V_\kk(x,v)$. This partly reflects the heavy-tailed property of the L\'evy noises. In particular, according to \eqref{E4},
    $$\mathcal V_{\kk,\theta}(x,v)\simeq (|v|^2+U(x))^{\theta/2}\quad \hbox{ as } |v|^2+U(x) \to \infty.$$
 \end{itemize}

 \end{remark}

 \begin{proof}[Proof of Proposition $\ref{pro1}$]
Below, we stipulate  $(x,v)\in\mathscr K$ and fix $\kk\in(0,\kk^*),$ which obviously implies $V_\kk\ge1.$ Since
$$
 \nn_x \mathcal V_{\kk,\theta}(x,v)=\frac{\theta}{2}V_\kk(x,v)^{{\theta}/{2}-1} \nn_x V_{\kk}( x, v),\quad  \nn_v \mathcal V_{\kk,\theta}(x,v)=\frac{\theta}{2}V_\kk(x,v)^{{\theta}/{2}-1} \nn_v V_{\kk}( x, v),
$$
we can write
\begin{align*}
(\mathscr L\mathcal V_{\kk,\theta})(x,v)&=\frac{\theta}{2}V_\kk(x,v)^{{\theta}/{2}-1}\big(\<\nn_xV_{\kk}( x,v),v\>-\<\nn_vV_{\kk}( x,v),\gamma  v+\nn U( x)\>\big)\\
&\quad+\sum_{i=1}^N\int_{\{|z |\le1\}}\big(\mathcal V_{\kk,\theta}(x,v+S_i(z ))-\mathcal V_{\kk,\theta}(x,v)-\<\nn_v^{(i)}\mathcal V_{\kk,\theta}(x,v),z \> \big)\,\nu^{(i)}(\d z )\\
&\quad+\sum_{i=1}^N \int_{\{|z |>1\}}\big(\mathcal V_{\kk,\theta}(x,v+S_i(z ))-\mathcal V_{\kk,\theta}(x,v) \big)\,\nu^{(i)}(\d z )\\
&=:\frac{\theta}{2}V_\kk(x,v)^{{\theta}/{2}-1}I_1(x,v)+I_2(x,v)+I_3(x,v).
\end{align*}

In the following, we will quantify  the terms $I_1,I_2$ and $I_3$, respectively.
First, since
 \begin{align*}
 \nn_x V_{\kk}( x,v)&=\nn U(x)+\frac{\kk}{|\nn U(x)|^2}\alpha'(U(x)) U(x)\<v,\nn U(x)\>\nn U(x)\\
 &\quad+\frac{ \kk \alpha(U(x))}{|\nn U(x)|^2}\bigg[   \<v,\nn U(x)\>\nn U(x) \\
 &\qquad\qquad\qquad\quad+ U(x)\bigg(\nn^2 U(x)-\frac{2 ((\nn^2U(x) \nn U(x))\otimes\nn U(x) ) }{|\nn U(x)|^2}\bigg)v\bigg],
 \end{align*}
 and
$$
 \nn_v V_{\kk}(x,v)=v+\frac{\kk  }{|\nn U(x)|^2}\alpha(U(x))U(x)\nn U(x),
$$
we deduce
that
\begin{align*}
 I_1(x,v)&=-\gamma|v|^2-\gamma\psi_\kk(x,v)\\
 &\quad+\bigg(-\kk \alpha(U(x))U(x)+\frac{ \kk }{|\nn U(x)|^2}\big(\alpha(U(x))+\alpha'(U(x))U(x)\big) \<v,\nn U(x)\>^2\bigg)\\
 &\quad+\frac{ \kk \alpha(U(x))U(x)}{|\nn U(x)|^2}  \bigg( \<v,\nn^2 U(x)v\> -\frac{ 2  }{|\nn U(x)|^2}   \<v,((\nn^2U(x) \nn U(x))\otimes\nn U(x)) v\>\bigg)\\
 &=:-\gamma|v|^2-\gamma\psi_\kk(x,v)+ I_{11}(x,v)+I_{12}(x,v).
 \end{align*}

By means of \eqref{E2}, it follows readily that
$$
-\gamma\psi_\kk(x,v)\le  \gamma\theta_0 \kk^2 U(x)\I_{\{U(x)\ge r_0\}}+\frac{\gamma}{4 }|v|^2.
$$
Next, because of $\alpha(u)=0$ for $u\le r_0$,  $\alpha(u)=1$ for $u\ge 2r_0$ and $|\alpha' (u)|\le  2/r_0$ for $u\in [r_0,2r_0]$, we find that
$$
	I_{11}(x,v)\le-\kk  \alpha(U(x))U(x)\I_{\{r_0\le U(x)<2r_0\}} -\kk  U(x)\I_{\{U(x)\ge 2r_0\}} + 5
	\kk    |v|^2.
$$

 Furthermore, by virtue of  $|\nn U(x)|\ge 1$ for $U(x)\ge r_0$,
 $\alpha(u)=0$ for $u\le r_0$,
 \eqref{E3*} as well as \eqref{E3}, we deduce that
\begin{align*}
		I_{12}(x,v)&\le \frac{ 3\kk }{|\nn U(x)|^2}      U(x)\|\nn^2 U(x)\|\cdot|v|^2  \I_{\{U(x)\ge \beta_0 \}} +\frac{
			3\kk \beta_0}{|\nn U(x)|^2} \|\nn^2 U(x)\|\cdot|v|^2\I_{\{r_0\le U(x)\le \beta_0 \}}\\
		&\le 3\kk \big(C_U^*  \vee   (
		\beta_0  C^*_{U,\beta_0})\big)|v|^2,
	\end{align*}
where $\beta_0:=(2r_0)\vee R_U^*$. Therefore, we arrive at
\begin{align*}
	I_1(x,v)\le& -\Big(\frac{3}{4}\gamma-5
	\kk-3\kk \big(C_U^*  \vee   (
	\beta_0  C^*_{U,\beta_0})\big)\Big)|v|^2\\&-\kk(1-\gamma\theta_0 \kk)  U(x)\I_{\{U(x)\ge 2r_0\}}+ \gamma\theta_0 \kk^2  U(x)\I_{\{r_0\le U(x)< 2r_0\}}.
	\end{align*}

Thanks to $\kappa\in(0,\kappa^*)$, we obtain from the definition of $\kappa^*$ given in \eqref{EE2*} that
$$
  5\kk+3\kk (C_U^*  \vee   (
 \beta_0 C^*_{U,\beta_0} ))\le \frac{ \gamma}{4},\qquad \frac{1}{2}-\gamma\theta_0 \kk\ge0
$$
so that
$$
	I_1(x,v)\le -\frac{\gamma}{2}|v|^2 -\frac{1}{2}\kk   U(x)\I_{\{U(x)\ge 2r_0\}}+ \gamma\theta_0 \kk^2  U(x)\I_{\{r_0\le U(x)< 2r_0\}}.
$$
As a result, combining this with \eqref{E4}, $V_\kk\ge1$ and $\theta\in(0,2]$ yields for some constants $c_1,C_1>0$
\begin{equation}\label{e:proof1}
\frac{\theta}{2}V_\kk(x,v)^{{\theta}/{2}-1}I_1(x,v)\le - c_1V_\kk(x,v)^{{\theta}/{2}}+C_1.
\end{equation}

Note that for each $i=1,\cdots,N,$
\begin{align*}
(\nn_v^{(i)})^2\mathcal V_{\kk,\theta}(x,v)&=\frac{\theta}{2}V_{\kk }(x,v)^{{\theta}/{2}-1}\Bigg[\I_{d\times d} + \Big(\frac{\theta}{2}-1\Big)V_{\kk }(x,v)^{ -1}\\
&\quad   \times\bigg(v^{(i)}+\frac{\kk\alpha(U(x))U(x)}{|\nn U(x)|^2}\nn^{(i)} U(x)\bigg)\otimes\bigg(v^{(i)}+\frac{\kk\alpha(U(x)) U(x)}{|\nn U(x)|^2}\nn^{(i)} U(x)\bigg)\Bigg],
\end{align*}
where the second matrix in the big bracket is non-positive definite. Consequently,
 the mean value theorem, together with the  prerequisite $\theta\in(0,2)$ and the fact that $V_{\kk}\ge1$, leads to
\begin{equation}\label{e:proof2}
I_2(x,v)\le\frac{\theta}{4}\sum_{i=1}^N\int_{\{|z|\le1\}}|z|^2\,\nu^{(i)}(\d z)<\8
\end{equation}
since $\nu^{(i)}(\d z)$ is a L\'{e}vy measure on $\R^d.$

 Next, making use  of the inequalities:  $|a^{{\theta}/{2}}-b^{{\theta}/{2}}|\le |a-b|^{{\theta}/{2}}$ for $a, b\ge0,$ and $2ab\le \vv a^2+b^2/\vv$
for all $a,b\ge0,\vv>0$ yields for all $\vv_1,\vv_2>0,$
 \begin{align*}
 I_3(x,v)
 &\le \sum_{i=1}^N  \int_{\{|z|>1\}}\Big|\<v^{(i)},z \>+ \frac{| z |^2}{2} +\frac{\kk\alpha(U(x)) U(x)}{|\nn U(x)|^2}\<z,\nn^{(i)} U(x)\> \Big|^{{\theta}/{2}}\,\nu^{(i)}(\d z)\\
 &\le \sum_{i=1}^N \int_{\{|z|>1\}}\bigg(|v^{(i)}|\cdot|z |+ \frac{| z |^2}{2} +\frac{\kk\alpha(U(x)) U(x)}{|\nn U(x)|}\I_{\{U(x)\ge r_0\}}|z|\bigg)^{{\theta}/{2}}\,\nu^{(i)}(\d z)\\
 &\le  \sum_{i=1}^N\int_{\{|z|>1\}}\bigg(\Big(\vv_1|v^{(i)}|^2+\frac{1}{4\vv_1}\Big) |z|+ \frac{| z|^2}{2} \\
 &\qquad\qquad\qquad\quad+\Big(\Big(\vv_2U(x)+\frac{\kk^2 U(x)}{4\vv_2|\nn U(x)|^2}\Big)\I_{\{U(x)\ge \beta_0\}}+\kk\beta_0\Big)|z|\bigg)^{{\theta}/{2}}\,\nu^{(i)}(\d z)\\
 &\le (\vv_1|v|^2+\vv_2U(x)\I_{\{U(x)\ge \beta_0\}})^{{\theta}/{2}}\sum_{i=1}^N\int_{\{|z|>1\}}|z|^{ {\theta}/{2} }\,\nu^{(i)}(\d z)\\
 &\quad+\left(\frac{1}{2} +\frac{1}{ 4\vv_1}   +\frac{\kk^2 C^*_U}{ 4\vv_2 }+\kk\beta_0\right)^{{\theta}/{2}}\sum_{i=1}^N\int_{\{|z|>1\}}|z|^\theta\,\nu^{(i)}(\d z),
\end{align*}
where in the second inequality we use     $|\nn^{(i)}U(x)|\le |\nn U(x)|$; in the third inequality we employ the facts
$|\nn U(x)|\ge 1$ for $x\in \R^{Nd}$ with $U(x)\ge r_0$ and $\alpha(u)=1$ for $u\ge 2r_0;$
and the last inequality holds true from \eqref{E3} and $|v^{(i)}|\le |v|$. Subsequently,
using \eqref{E4} and Assumption $({\bf H}_\nu)$, and choosing $\vv_1$ and $\vv_2$ small enough, we can get that for some constant $C_2>0,$
\begin{equation}\label{e:proof3}
I_3(x,v)\le \frac{c_1}{2}V_\kk(x,v)^{ {\theta}/{2}}+C_2,
\end{equation} where $c_1$ is given in \eqref{e:proof1}.

Finally, the assertion \eqref{E6} follows by combining \eqref{e:proof1} with \eqref{e:proof2} and \eqref{e:proof3}.
\end{proof}

In the sequel, we set an example to show the  Assumption  $({\bf H}_U)$ is verifiable.   In particular,
the following example demonstrates that the singular potentials involved in \eqref{E1} include
the Lennard-Jones type potentials.

\begin{example}\label{ex1}
 Consider the Lennard-Jones type potential (see e.g. \cite[Example 4.4]{HM})
$$
U(x)=\sum_{i=1}^NU_c(x^{(i)})+\sum_{1\le i<j\le N}U_I(x^{(i)}-x^{(j)}),
$$
where
$$
U_c(u):=A
(1+|u|^2)^{\alpha/2}+\phi_c(u),\quad u\in\R^d; \quad U_I(u):=\frac{B}{|u|^\beta}+\phi_I(u),\quad u\in\mathscr D(U_I)
$$
with $A,B,\beta>0,\alpha\ge2$,
$\phi_c\in C^\8(\R^d)$ and $\phi_I\in C^\infty(\mathscr D(\phi_I))$ so that
\begin{align}
\lim_{|u|\to \8}|u|^{-\alpha}|\phi_c(u)|&=\lim_{|u|\to \8}|u|^{1-\alpha}|\nn\phi_c(u)|=\lim_{|u|\to \8}|u|^{2-\alpha}\|\nn^2\phi_c(u)\|=0,\label{EE3*}\\
\lim_{u\in\mathscr D(\phi_I),|u|\to0}|u|^\beta|\phi_I(u)|&= \lim_{u\in\mathscr D(\phi_I),|u|\to0}|u|^{1+\beta}|\nn\phi_I(u)|=\lim_{u\in\mathscr D(\phi_I),|u|\to0}|u|^{2+\beta}\|\nn^2\phi_I(u)\|=0,
\label{EE4*}
\end{align}
and that  for some $r>0$,
\begin{equation}\label{EE5*}
\phi_I,\,\nn\phi_I,\, \nn^2\phi_I\,\mbox{ are bounded on the set }  \{u\in\mathscr D(\phi_I): |u|\ge r\}.
\end{equation}
In particular, when $\phi_c(u)=0, \beta=12$ and $\phi_I(u)=-\frac{C}{|u|^6}$
for some positive constant $C$, the potential $U$ above corresponds to the classical  Lennard-Jones potential.

It is clear that $U: \mathscr D(U) \to \R
$ is a $C^\infty$-function. On the other hand, note that for any non-zero vector $u\in\R^d,$
\begin{equation}\label{e:add1}
\begin{split}
U_c(u)
&\ge \frac{1}{2} A|u|^\alpha +\frac{1}{2}A|u|^\alpha\Big(1-\frac{2}{A}|u|^{-\alpha}|\phi_c(u)|\Big), \\
U_I(u)&\ge  \frac{B}{2|u|^\beta}+\frac{B}{2|u|^\beta}\Big(1-
\frac{2}{B}|u|^\beta|\phi_I(u)|\Big).
\end{split}
\end{equation}
Then, by making use of
\begin{equation}\label{EE6*}
\lim_{|u|\to \8}|u|^{-\alpha}|\phi_c(u)|=\lim_{u\in\mathscr D(\phi_I),|u|\to0}|u|^\beta|\phi_I(u)|= 0,
\end{equation}
due to \eqref{EE3*} and \eqref{EE4*},
in addition to the   boundedness of
$\phi_c$ and $\phi_I$ in a closed ball and out of a closed ball (see \eqref{EE5*}), respectively,  there exists a constant $C_U^\star>0$ such that
$$\bar U(x):=U(x)+C_U^\star\ge1.$$ In particular, the function $U(x)$ is bounded from below, and $U(x)\to +\8$ if and only if  $x\to\partial \mathscr D(U) $
or $|x|\to+\8$, thanks to \eqref{e:add1}.

According to the conclusions above, in order to prove \eqref{E3-} it amounts  to verify \eqref{E3} holds for the function $\bar U$.
Recall from \cite[Lemma A.1]{HM} that there exist constants $C_1,C_2,C_3>0$ such that
$$
|\nn \bar U(x)|\ge C_1|x|^{\alpha-1}+C_2\sum_{1\le i<j\le N} |x^{(i)}-x^{(j)}|^{-\beta-1}-C_3.
$$
Recall also that $\bar U(x)\to \8$ if and only if $|x|\uparrow\8$ or $|x^{(i)}-x^{(j)}|\downarrow 0$ for some $i\neq j$. Therefore, there exists a constant $R_{\bar U}^*>0$ such that for all $x\in\mathscr D(\bar U)$ with $\bar U(x)\ge R_{\bar U}^*,$
$$
\frac{1}{2}C_1|x|^{\alpha-1}+\frac{1}{2}C_2\sum_{1\le i<j\le N} |x^{(i)}-x^{(j)}|^{-\beta-1}\ge C_3
$$
so that for all $x\in\mathscr D(\bar U)$ with $\bar U(x)\ge R_{\bar U}^*,$
\begin{equation}\label{EE7*}
|\nn \bar U(x)|\ge F(x):=\frac{1}{2}C_1|x|^{\alpha-1}+\frac{1}{2}C_2\sum_{1\le i<j\le N} |x^{(i)}-x^{(j)}|^{-\beta-1}.
\end{equation}
By taking \eqref{EE6*} into consideration, along with the boundedness of $\phi_c$ and $\phi_I$ in a closed ball and outside   a closed ball, respectively, there is a constant  $C_4>0 $ such that
\begin{equation}\label{EE8*}
\bar U(x)\le C_4\bigg(1+|x|^\alpha+\sum_{1\le i<j\le N}|x^{(i)}-x^{(j)}|^{-\beta }\bigg).
\end{equation}
This, besides \eqref{EE7*}, implies that
for all $x\in\mathscr D(\bar U)$ satisfying $\bar U(x)\ge R_{\bar U}^*$,
\begin{align*}
\frac{\bar U(x)}{|\nn \bar U(x)|^2}&\le \frac{ 4C_4\big(1+|x|^\alpha+\sum_{1\le i<j\le N}|x^{(i)}-x^{(j)}|^{-\beta }\big)}{C_1^2|x|^{2(\alpha-1)}+C_2^2(\sum_{1\le i<j\le N}|x^{(i)}-x^{(j)}|^{-\beta-1 })^2}\\
&\le \frac{ 4C_4\big(1+(N-1)N/2+|x|^\alpha+\sum_{1\le i<j\le N}|x^{(i)}-x^{(j)}|^{-\beta }\I_{\{|x^{(i)}-x^{(j)}|\le 1\}}\big)}{C_1^2|x|^{2(\alpha-1)}+C_2^2(\sum_{1\le i<j\le N}|x^{(i)}-x^{(j)}|^{-\beta-1 }\I_{\{|x^{(i)}-x^{(j)}|\le 1\}})^2}.
\end{align*}
By the fundamental inequality: for any $a,b,c,d>0,$
$$
\frac{a+b}{c+d}\le \frac{a}{c}\vee\frac{b}{d},
$$
and the fact that $\alpha-2(\alpha-1)=-\alpha+2\le 0$, yields that there exists a constant $C_{\bar U}^*>0$ such that
\begin{equation}\label{E10*}
\sup_{x\in\mathscr D(\bar U): \bar U(x)\ge R_{\bar U}^*}\frac{\bar U(x)}{|\nn \bar U(x)|^2}\le C_{\bar U}^*.
\end{equation}

Next, applying from \eqref{EE3*} and \eqref{EE4*}
$$
\lim_{|u|\to \8}|u|^{2-\alpha}\|\nn^2\phi_c(u)\|=\lim_{u\in\mathscr D(\phi_I),|u|\to0}|u|^{2+\beta}{\|\nn^2\phi_I(u)\|}=0
$$
and taking advantage of the definition of $U(x)$  enables us to obtain that there is a   constant $C_5>0 $ such that for all $x\in\mathscr D(\bar U)$,
\begin{equation}\label{E9*}
\| \nn^2\bar U(x)\|\le C_5 \bigg(1+|x|^{\alpha-2}+\sum_{1\le i<j\le N}|x^{(i)}-x^{(j)}|^{-\beta-2}\bigg).
\end{equation}
Subsequently, combining \eqref{EE7*} with \eqref{EE8*} and \eqref{E9*}, in addition to $\alpha+\alpha-2-2(\alpha-1)=0$ and $-\beta-(\beta+2)+2(\beta+1)=0,$
we infer that there exists a constant $C_{\bar U}^{**}>0$ such that for all $x\in\mathscr D(\bar U)$ satisfying $\bar U(x)\ge R_{\bar U}^*$,
$$
\sup_{x\in\mathscr D(\bar U): \bar U(x)\ge R_{\bar U}^*}\frac{\bar U(x)\|\nn^2 \bar U(x)\|}{|\nn \bar U(x)|^2}\le C_{\bar U}^{**}.
$$
Hence, we conclude that \eqref{E3} holds true for $\bar U$ by taking \eqref{E10*} into account.

Therefore, based on all the conclusions above, we infer that $U(x)$ satisfies the Assumption  $({\bf H}_U)$.
\end{example}

Example \ref{ex1} indicates that the Lyapunov condition \eqref{E6} is available  for a wide range of singular potentials, including the Lennard-Jones potential, the Riesz potential
(i.e., $U_I(u)=|u|^{1-d}$ for all $d\ge 2$), and the (Newtonian)  Coulomb potential (i.e., $U_I(u)=|u|^{2-d}$ for all $d\ge 3$) as typical candidates.

\subsection{Case 2}
Note that the Coulomb type potential for the case  $d=2$ (for example, $U_I(u)=-\log |u|$ for non-zero vectors $u\in\R^2$ does not satisfy \eqref{EE5*})
 is excluded by Proposition \ref{pro1}. In order to handle the Coulomb type potential, as another  classical  representative of singular potentials, we shall put forward
 another collection of sufficient conditions so that the Lyapunov type condition \eqref{E6} remains true.

To this end, suppose that {\it the potential $U$ can be written as
\begin{equation}\label{EE1}
U(x)=\sum_{i=1}^NV(x^{(i)})+\frac{1}{N}\sum_{ i,j=1, j\neq i}^N K(x^{(i)}-x^{(j)}),
\end{equation}
where $V:\R^d\to \R$ is a $C^\infty$-function and $K:\mathscr D(K)\to\R$ is a radial $C^\infty$-function so that $\{x\in\R^d: |x|\ge r \}\subset  \mathscr D(K):=\{x\in \R^d: |K(x)|<\infty\}$ and
$\sup_{|x|\ge r}|\nabla K(x)|<\infty$ for any $r>0$.
 Moreover, the following two conditions are satisfied for $V$ and $K$, respectively.}
\begin{itemize} \it

 \item[$({\bf H}_V)$] there exist constants $C_{V}^*,C_V^{**}>0$ and $M_V,M_V^*\ge0$ such that for all $x\in\R^d,$
\begin{equation}\label{*E1}
C_{V}^*|x|^2-M_{V}\le
 V(x)\le \frac{1}{C_{V}^{*}}\big( M_{V}+\<\nn V(x),x\> \big),
\end{equation}
and
\begin{equation}\label{EE2}
|\nn V(x)|\le C_{V}^{**}V(x)+M_{V}^*.
\end{equation}

\item[$({\bf H}_K)$] there exist constants $ R_{K}, C_{K}^*>0$ such that for all $x\in\mathscr D({K})$ with $|x|\le R_{K} $,
\begin{equation}\label{*E5}
K(x)\ge0,\quad  \frac{1}{|x|}\<x,\nn K(x)\>\le -C_{K}^* K(x) .
\end{equation}
Moreover,  there exists a constant $C_{K}^{**}>0$ such that for all $x\in \mathscr D({K})$,
\begin{equation}\label{EE3}
\sum_{i,j,k=1,j,k\neq i}^N \<{\bf n}( x^{(i)}-x^{(j)}),\nn  K(x^{(i)}-x^{(k)}) \>\le C_{K}^{**}\sum_{i,j=1,j\neq i}^N \<{\bf n}( x^{(i)}-x^{(j)}),   \nn  K(x^{(i)}-x^{(j)}) \>,
\end{equation} where ${\bf n}(u):=u/|u|$ for a non-zero vector $u\in\R^d.$
\end{itemize}

\ \

Note that, in the present setting $\mathscr K:=\mathscr D(U)\times\R^{Nd}=\mathscr D(K)\times\R^{Nd}$. For $\alpha,\beta, C^\star>0,$ define
\begin{equation}\label{EE}
\begin{split}
\mathscr V_{\alpha,\beta}(x,v)&=C^\star+\frac{1}{2}| v|^2+U(x)-\frac{\alpha}{N}\sum_{i,j=1,  j\neq i}^N \<v^{(i)},{\bf n}(x^{(i)}-x^{(j)})\> +\beta\<x,v\>,\quad (x,v)\in\mathscr K.
\end{split}
\end{equation}
By the H\"older inequality and the Cauchy–Schwarz inequality, one has
$$
\Big|-\frac{\alpha}{N}\sum_{i,j=1,  j\neq i}^N \<v^{(i)},{\bf n}(x^{(i)}-x^{(j)})\> +\beta\<x,v\>\Big|\le \frac{1}{2}(\alpha+\beta)|v|^2+\frac{1}{2}\beta|x|^2+\frac{1}{2}N\alpha.
$$

Note that $K:\mathscr D(K)\to\R$ is a radial $C^\infty$-function so that $\{x\in\R^d: |x|\ge r \}\subset  \mathscr D(K):=\{x\in \R^d: |K(x)|<\infty\}$ and
\begin{equation}\label{*E6}
C_{K,r}:= \sup_{|x|\ge r}|\nabla K(x)|<\infty
\end{equation} for every $r>0$.
By virtue of   the local boundedness of $V$ and the radial property of $K(\cdot)$, it follows  from \eqref{*E6} that for any $r_1,r_2> 0$,
\begin{align*}
  &\sum_{i=1}^N|V(x^{(i)})|\I_{\{|x^{(i)}|\le r_1\}}+\frac{1}{N}\sum_{i,j=1,j\neq i}^N |K(x^{(i)}-x^{(j)})|\I_{\{|x^{(i)}-x^{(j)}|\ge r_2\}}\\
  &\le N\max_{|u|\le r_1}|V(u)|+N|K(r_2)| + C_{K,r_2}  N (2|x|-r_2).
  \end{align*}
Thus, for any $r_1>(M_V/C_V^*)^{1/2}$ and $r_2\in(0,R_K)$, we infer from \eqref{*E1} that
\begin{align*}
\mathscr V_{\alpha,\beta}(x,v)&\le C^\star+\frac{1}{2}(1+\alpha+\beta)| v|^2+2C_{K,r_2}  N |x|\\
&\quad+\Big(1+\frac{\beta}{2C_V^*}\Big)\sum_{i=1}^N V(x^{(i)}) \I_{\{|x^{(i)}|\ge r_1\}}+\frac{1}{N}\sum_{i,j=1,j\neq i}^N  K(x^{(i)}-x^{(j)}) \I_{\{|x^{(i)}-x^{(j)}|\le r_2\}}\\
&\quad+ \Big(1+\frac{\beta  }{2C_V^*}  \Big)N \max_{|u|\le r_1}|V(u)| +\Big(\frac{\beta M_V }{2C_V^*} +\frac{\alpha}{2}+|K(r_2)|-C_{K,r_2}r_2\Big)N
 \end{align*}
and
\begin{align*}
\mathscr V_{\alpha,\beta}(x,v)&\ge C^\star+\frac{1}{2}(1-(\alpha+\beta))| v|^2+\frac{1}{2}(C_{V}^*-\beta)\sum_{i=1}^N |x^{(i)}|^2-2C_{K,r_2} N |x| \\
&\quad+\frac{1}{2}\sum_{i=1}^NV(x^{(i)})\I_{\{|x^{(i)}|\ge r_1\}}+\frac{1}{N}\sum_{ i,j=1, j\neq i}^N K(x^{(i)}-x^{(j)})\I_{\{|x^{(i)}-x^{(j)}|\le r_2\}} \\
&\quad-\Big(\frac{1}{2} \max_{|u|\le r_1}|V(u)|+|K(r_2)|-C_{K,r_2}r_2+\frac{1}{2}(\alpha+M_V) \Big)N.
 \end{align*}
In the sequel, we shall take $\alpha,\beta>0$ satisfying
\begin{equation}\label{*E4}
 \alpha  \le  
 \frac{\beta C_V^*}{2C_V^{**}},\qquad  \beta \le \frac{1}{2}(C^*_V\wedge\gamma)\wedge \frac{(C_V^*)^2}{8 \gamma}\wedge\frac{C_V^{**}}{2C_V^{**}+C_V^*}
\end{equation}
so that
$$
\alpha+\beta\le \frac{1}{2},\qquad \beta\le \frac{1}{2}C^*_V.
$$
Therefore, there exists a constant
$C^\star>0$ large enough such that
$$\mathscr V_{\alpha,\beta}(x,v)\ge 1,\quad (x,v)\in\mathscr K;$$
moreover, for any $r_1>(M_V/C_V^*)^{1/2}$ and $r_2\in(0,R_K)$,
 there are constants $0<c_1\le c_2$ (both are dependent on $r_1,r_2$) such that
\begin{equation}\label{*E10}
c_1\le\frac{\mathscr V_{\alpha,\beta}(x,v)}{1+|v|^2+\sum_{i=1}^NV( x^{(i)})\I_{\{|x^{(i)}|\ge r_1\}}+\frac{1}{N}\sum_{i,j=1,j\neq i}^NK(x^{(i)}-x^{(j)})\I_{\{|x^{(i)}-x^{(j)}|\le r_2\}}}\le c_2
\end{equation}
 holds for all $(x,v)\in\mathscr K.$

Before moving forward, let us make some comments on the preceding assumptions   and the construction of the function $\mathscr V_{\alpha,\beta}$.
\begin{remark}
\begin{itemize}
\item[{\rm(i)}] The decomposition of the potential $U(x)$ given in \eqref{EE1} has physical meanings. The term with the function $V$ stands for the confining potential due to the external forces, and in this sense Assumption $({\bf H}_V)$ is natural in the literature; see e.g. \cite{MSH,Wu}. On the other hand, the term expressed by the function $K$ stands for the interaction potential. In particular, the second condition in \eqref{*E5} roughly indicates that the repulsive forces of the interaction will produce the dissipation when the particles approach each other, and \eqref{EE3} shows that the interaction among the particles enjoys some homogeneous property.

\item[{\rm (ii)}] The $\<x,v\>$-perturbation term in the function $\mathscr V_{\alpha,\beta}$ has been frequently used in the construction of Lyapunov function for Langevin dynamics (see \cite{Wu} and \cite{LM} for regular and singular cases, respectively). In particular, from the estimates above, we can see that
    $$\mathscr V_{\alpha,\beta}(x,v)\simeq  |v|^2+U(x) \quad \hbox{ as } |v|^2+U(x) \to \infty.$$ Hence, the Lyapunov function $\mathcal V_{\alpha,\beta,\theta}(x,v)$ in Proposition \ref{pro2} below fulfills that
    $$\mathcal V_{\alpha,\beta,\theta}(x,v)\simeq (|v|^2+U(x))^{\theta/2}\quad \hbox{ as } |v|^2+U(x) \to \infty.$$     \end{itemize}

\end{remark}

The proposition below illustrates that the Lyapunov condition \eqref{E6} is still valid under another set of sufficient conditions and, most importantly,  is applicable to the Coulomb type potential.

\begin{proposition}\label{pro2}
Assume that $({\bf H}_V)$, $({\bf H}_K)$ and $({\bf H}_\nu)$ hold. Then,
concerning  $\mathcal V_{\alpha,\beta,\theta}(x,v):=\mathscr V_{\alpha,\beta}(x,v)^{{\theta}/{2}}$ with positive
$\alpha,\beta$  given in \eqref{*E4}, there exist constants $\lambda_{\mathcal V},C_{\mathcal V}>0$ such that
\begin{equation}\label{EE6}
  (\mathscr L\mathcal V_{\alpha,\beta,\theta})(x,v)\le-\lambda_{\mathcal V} \mathcal V_{\alpha,\beta,\theta}(x,v)+C_{\mathcal V},\quad (x,v)\in\mathscr K,
\end{equation}
where the infinitesimal generator $\mathscr L$ is defined as in \eqref{E9}.
\end{proposition}

 \begin{proof}
 In the sequel, we fix $(x,v)\in\mathscr K$.
 According to the definition of $\mathscr L$, defined in \eqref{E9},
we   obtain  that
\begin{align*}
(\mathscr L\mathcal V_{\alpha,\beta,\theta})(x,v)&=\frac{\theta}{2}\mathscr V_{\alpha,\beta}(x,v)^{{\theta}/{2}-1}\big(\<\nn_x \mathscr V_{\alpha,\beta}( x, v),v\>-\<\nn_v  \mathscr V_{\alpha,\beta}( x,v),\gamma  v+\nn  U ( x)\>\big)\\
&\quad+\sum_{i=1}^N\int_{\{|z |\le1\}}\big(\mathcal V_{\alpha,\beta,\theta}(x,v+S_i(z ))-\mathcal V_{\alpha,\beta,\theta}(x,v)-\<\nn_v^{(i)}\mathcal V_{\alpha,\beta,\theta}(x,v),z \> \big)\,\nu^{(i)}(\d z )\\
&\quad+\sum_{i=1}^N \int_{\{|z |>1\}}\big(\mathcal V_{\alpha,\beta,\theta}(x,v+S_i(z ))-\mathcal V_{\alpha,\beta,\theta}(x,v) \big)\,\nu^{(i)}(\d z )\\
&=:\frac{\theta}{2}\mathscr V_{\alpha,\beta}(x,v)^{{\theta}/{2}-1}I_1(x,v)+I_2(x,v)+I_3(x,v).
\end{align*}
Hereinafter, it boils down  to estimating  the quantities $I_1,I_2$  and $I_3$, respectively, in order to achieve \eqref{EE6}.

Owing to ${\bf n}(-u)=-{\bf n}(u)$ for a non-zero vector $u\in\R^d,$ it is easy to see that
$$
\sum_{i,j=1,j\neq i}^N \<v^{(i)},{\bf n}(x^{(i)}-x^{(j)})\>
=\frac{1}{2}\sum_{i,j=1,j\neq i}^N\<v^{(i)}-v^{(j)},{\bf n}(x^{(i)}-x^{(j)})\>.
$$
Consequently,
 $\mathscr V_{\alpha,\beta}$  can be reformulated  as
$$
\mathscr V_{\alpha,\beta}(x,v)=C^\star+\frac{1}{2}| v|^2+U(x) -\frac{\alpha}{2N}\sum_{i,j=1,j\neq i}^N \<v^{(i)}-v^{(j)},{\bf n}(x^{(i)}-x^{(j)})\> +\beta\<x,v\>.
$$
With this at hand, we have
 \begin{align*}
 \nn_x^{ (i)  } \mathscr V_{\alpha,\beta}(x,v)&=\nn^{(i)}_x U( x)+\beta v^{(i)} -\frac{\alpha}{2N}\Theta(x^{(i)},v),\\
 \nn_v^{(i)} \mathscr V_{\alpha,\beta}(x,v)& =v^{(i)}-\frac{\alpha  }{2N}\sum_{j=1, j\neq i}^N {\bf n}( x^{(i)}-x^{(j)}) +\beta x^{(i)},
 \end{align*}
 where
 $$\Theta_i(x,v):= \sum_{j=1, j\neq i}^N\frac{1}{|x^{(i)}-x^{(j)}|} \big( \I_{d\times d}-({\bf n}(x^{(i)}-x^{(j)}) \otimes {\bf n}(x^{(i)}-x^{(j)})\big) (v^{(i)}-v^{(j)}).$$
 Thus,
we derive that
\begin{align*}
 I_1(x,v)&=\sum_{i=1}^N\<\nn_x^{ (i)  } \mathscr V_{\alpha,\beta}(x,v),v^{(i)}\>-\sum_{i=1}^N\<\nn_v^{ (i)  } \mathscr V_{\alpha,\beta}(x,v),\gamma  v^{(i)}+\nn^{(i)}_x U( x)\>\\
 &=-(\gamma-\beta)|v|^2-\frac{\alpha}{2N}\sum_{i=1}^N\<\Theta_i(x,v),v^{(i)}\>\\
 &\quad+\bigg(\frac{\alpha \gamma }{2N} \sum_{i,j=1, j\neq i}^N\< {\bf n}( x^{(i)}-x^{(j)}) ,   v^{(i)}\>-\beta\gamma\sum_{i=1}^N\<  x^{(i)},  v^{(i)}\>\bigg)\\
 &\quad+\bigg(\frac{\alpha}{2N}\sum_{i,j=1,j\neq i}^N\<  {\bf n}( x^{(i)}-x^{(j)}) ,\nn^{(i)}_x U( x)\> -\beta \sum_{i=1}^N\<x^{(i)},\nn^{(i)}_x U( x)\>  \bigg)\\
 &=:-(\gamma-\beta)|v|^2-\frac{\alpha}{2N}{  I_{11}(x,v)}+ I_{12}(x,v) +I_{13}(x).
\end{align*}

Via the Cauchy–Schwarz inequality, besides $\Theta_i(x,-v)=-\Theta_i(x,v)$,
 it is obvious that
  \begin{align*}
  I_{11}(x,v)
 =\frac{1}{2}\sum_{i,j=1,j\neq i}^N\frac{1}{|x^{(i)}-x^{(j)}|^3}\big(|v^{(i)}-v^{(j)}|^2|x^{(i)}-x^{(j)}|^2-\<v^{(i)}-v^{(j)},x^{(i)}-x^{(j)}\>^2 \big)\ge0.
 \end{align*}
Applying  H\"older's inequality and
Cauchy–Schwarz's inequality, in addition to the basic inequality: $2ab\le a^2+b^2$ for $a,b\in\R,$  yields
\begin{align*}
   I_{12}(x,v)
   &\le \frac{2\beta\gamma^2}{(C_V^*)^2}|v|^2+\frac{1}{4}\beta (C_V^*)^2|x|^2+\frac{1}{16\beta}(\alpha C_V^*)^2N \\
   &\le \frac{2\beta\gamma^2}{(C_V^*)^2}|v|^2+\frac{1}{4}\beta  C_V^* \sum_{i=1}^N V(x^{(i)})+\frac{1}{4}N\beta  C_V^* M_V+\frac{1}{16\beta}(\alpha C_V^*)^2N,
   \end{align*}
 where the last display is due to \eqref{*E1}.
Note that
\begin{align*}
	\nn^{(i)}_x U(x) &= \nn V( x^{(i)}) +  \frac{1}{N}\sum_{k=1,k\neq i}^N \nn^{(i)} K(x^{(k)}-x^{(i)})+  \frac{1}{N}\sum_{l=1,l\neq i}^N \nn^{(i)} K(x^{(i)}-x^{(l)})\\
	&= \nn V( x^{(i)}) +  \frac{2}{N}\sum_{k=1,k\neq i}^N \nn^{(i)} K(x^{(i)}-x^{(k)}).
	\end{align*}
Since $K(\cdot)$ is  a radial function, we have $\nn^{(i)}K(x^{(i)}-x^{(j)})=-\nn^{(j)}K(x^{(j)}-x^{(i)})$, which implies
$$
 \frac{\beta}{N}\sum_{i,j=1,j\neq i}^N \<x^{(i)},\nn^{(i)} K(x^{(i)}-x^{(j)})\>=\frac{\beta}{2N}\sum_{i,j=1,j\neq i}^N \<x^{(i)}-x^{(j)},\nn^{(i)} K(x^{(i)}-x^{(j)})\>.
$$
 This
  fact, together with the expression of
 $U$, given in \eqref{EE1}, enables us to obtain from \eqref{*E1} and \eqref{EE2} that
\begin{equation}\label{EE4}
	\begin{split}
		I_{13}(x)&=-\beta\sum_{i=1}^N\<x^{(i)},\nn V(x^{(i)})\>+\frac{\alpha}{2N}\sum_{i,j=1,j\neq i}^N \<{\bf n}( x^{(i)}-x^{(j)}),\nn V(x^{(i)})\>\\
		&\quad-\frac{\beta}{N}\sum_{i,j=1,j\neq i}^N \<x^{(i)}-x^{(j)},\nn^{(i)} K(x^{(i)}-x^{(j)})\>\\
		&\quad+\frac{\alpha}{N^2}\sum_{i,j,k=1,j,k\neq i}^N \<{\bf n}( x^{(i)}-x^{(j)}),\nn^{(i)} K(x^{(i)}-x^{(k)}) \>\\
		&\le-\big(\beta C_{V}^{*}-\alpha C_{V}^{**}/2\big)\sum_{i=1}^NV(x^{(i)}) +(\beta M_V +M_V^*\alpha/2)N +\frac{1}{N}\sum_{i,j=1,j\neq i}^NJ(x^{(i)}-x^{(j)}),
	\end{split}
\end{equation}
where for a non-zero vector $u\in\R^d,$
$$J(u):=\Big(\frac{\alpha C_K^{**}}{N|u|}-\beta\Big)\<u,\nn K(u)\>.$$
Let $r^*= R_K\wedge\frac{\alpha C_K^{**}}{2N\beta}$. Then,
for any non-zero vector $u\in\R^d,$ we obtain from \eqref{*E6} that
\begin{align*}
J(u)&=
\Big(\frac{\alpha C_K^{**}}{N|u|}-\beta\Big)\<u,\nn K(u)\>\I_{\{|u|\le r^*\}}+\Big(\frac{\alpha C_K^{**}}{N|u|}-\beta\Big)\<u,\nn K(u)\>\I_{\{|u|\ge r^*\}}\\
&\le \frac{\alpha C_K^{**}}{2N|u|} \<u,\nn K(u)\>\I_{\{|u|\le r^*\}}+\Big(\frac{\alpha C_K^{**}}{N}+\beta|u|\Big)C_{K,r*}\I_{\{|u|\ge r^*\}},
\end{align*} where in the  inequality above we utilize the fact that $\frac{\alpha C_K^{**}}{2N|u|}-\beta\ge0$ whenever $|u|\le \frac{\alpha C_K^{**}}{2N\beta}$.
In particular, with the help of \eqref{*E1} and \eqref{*E5}, we deduce that there exists a constant $C_0>0$ such  that for  $u=x^{(i)}-x^{(j)}$,
\begin{align*}
		J(u)		&\le \frac{\alpha C_K^{**}}{2N|x^{(i)}-x^{(j)}|} \<x^{(i)}-x^{(j)},\nn K(x^{(i)}-x^{(j)})\>\I_{\{|x^{(i)}-x^{(j)}|\le r^*\}}\\
&\quad+\Big(\frac{\alpha C_K^{**}}{N}+\beta|x^{(i)}-x^{(j)}|\Big)C_{K,r*}\I_{\{|x^{(i)}-x^{(j)}|\ge r^*\}}\\
		&\le \frac{\alpha C_K^{**}}{2N|x^{(i)}-x^{(j)}|} \<x^{(i)}-x^{(j)},\nn K(x^{(i)}-x^{(j)})\>\I_{\{|x^{(i)}-x^{(j)}|\le r^*\}}+\Big(\frac{\alpha C_K^{**}}{N}+\beta|x^{(i)}|+\beta|x^{(j)}|\Big)C_{K,r*}\\
		&\le -\frac{\alpha C_K^*C_K^{**}}{2N}\ K(x^{(i)}-x^{(j)})\I_{\{|x^{(i)}-x^{(j)}|\le r^*\}}+\frac{1}{8}\beta C_V^*V(x^{(i)})+\frac{1}{8}\beta C_V^*V(x^{(j)})+C_0,
	\end{align*}
 Now, plugging the previous estimate on $J(\cdot)$ into the evaluation on $I_{13}$, given in \eqref{EE4}, we find that
\begin{align*}
		I_{13}(x)
		&\le -\Big(\frac{3}{4}\beta C_{V}^{*}
-\frac{1}{2}\alpha C_{V}^{**} \Big)\sum_{i=1}^NV(x^{(i)})  -\frac{\alpha C_K^*C_K^{**}}{2N^2}\sum_{i,j=1,j\neq i}^NK(x^{(i)}-x^{(j)})\I_{ \{|x^{(i)}-x^{(j)}|\le r^*\}}+C_1
	\end{align*}
for some constant  $  C_1>0.$

Consequently, in accordance with the estimates on $I_{11},I_{12},I_{13},$ we derive that for some constant $C_2>0,$
\begin{align*}
		I_1(x,v)&\le -\Big(\gamma-\beta-\frac{2\beta\gamma^2}{(C_V^*)^2}\Big)|v|^2   -\frac{1}{2} \big( \beta C_{V}^{*}
- \alpha C_{V}^{**} \big)\sum_{i=1}^NV(x^{(i)})\\
		& \quad-\frac{\alpha C_K^*C_K^{**}}{2N^2}\sum_{i,j=1,j\neq i}^NK(x^{(i)}-x^{(j)})\I_{ \{|x^{(i)}-x^{(j)}|\le r^*\}}+C_2.
	\end{align*}
According to the alternatives of $\alpha,\beta>0$, introduced in \eqref{*E4}, we have
$$
\gamma-\beta-\frac{2\beta\gamma^2}{(C_V^*)^2}\ge\frac{\gamma}{4},\quad \frac{1}{2} \big( \beta C_{V}^{*}
- \alpha C_{V}^{**} \big)\ge\frac{1}{4}\beta C_{V}^{*}.
$$
Subsequently,
 by taking \eqref{*E10} into consideration, there exist constants
$c_0, C_3>0$ such that
\begin{equation}\label{e:proof11}
 I_1(x,v)
\le -c_0 \mathcal V_{\alpha,\beta,\theta}(x,v)+C_3.
\end{equation}

By following exactly the line to derive \eqref{e:proof2}, we have
\begin{equation}\label{e:proof12}
I_2(x,v)\le\frac{\theta}{4}\sum_{i=1}^N\int_{\{|z|\le1\}}|z|^2\,\nu^{(i)}(\d z)<\8.
\end{equation}
 Next, applying the inequality: $|a^\kk-b^\kk|\le |a-b|^\kk$ for $a, b\ge0 $ and $\kk\in(0,1]$, we find by the Young inequality and \eqref{*E10} that
 \begin{equation}\label{e:proof13}
	\begin{split}
		I_3(x,v)&\le  \sum_{i=1}^N \int_{\{|z|>1\}}\Big|\<v^{(i)},z\>+ \frac{1}{2}| z|^2-\frac{\alpha}{N}\sum_{k=1,k\ne  i}^N \< z,{\bf n}(x^{(i)}-x^{(k)})\>+\beta\<x^{(i)},z\> \Big|^{{\theta}/{2}}\,\nu^{(i)}(\d z)\\
		&\le \sum_{i=1}^N\int_{\{|z|>1\}}\big( | z|^2 + (\alpha +|v^{(i)}|+\beta|x^{(i)}|)|z|\big)^{{\theta}/{2}}\, \nu^{(i)}(\d z)\\
		&\le  (1+\alpha +|v|+\beta|x|)^{{\theta}/{2}}\sum_{i=1}^N\int_{\{|z|>1\}}|z|^\theta\,\nu^{(i)}(\d z)\\
		&\le \frac{c_0}{2} \mathcal V_{\alpha,\beta,\theta}(x,v)+C_4
	\end{split}
\end{equation}
for some constant $C_4>0,$ where  the constant $c_0$ is that given in \eqref{e:proof11}.

Finally, the desired assertion \eqref{EE6} is available by taking  \eqref{e:proof11},  \eqref{e:proof12} and \eqref{e:proof13} into account.
\end{proof}

Before we end this section,
we present an example to explain the applicability of Proposition \ref{pro2} to the Coulomb potential (see e.g. \cite{LM}).

\begin{example}\label{ex2}   Let
 \begin{equation}\label{EE7}
 V(x)= A(1+|x|^2)^{\aa/2};\quad K(x)=-\log|x|\quad {\rm if}\quad d=2,\quad K(x)=|x|^{2-d} \quad {\rm if} \quad d\ge3,
  \end{equation}
where $\alpha\ge2  $
and $A>0$.

Obviously, $V\in C^\infty(\R^d)$,
$\<\nn V(x),x\>=A\alpha (1+|x|^2)^{\alpha/2-1}|x|^2$, and $|\nabla V(x)|\le A\alpha (1+|x|^2)^{\alpha/2-1}|x|$ for all $x\in \R^d$. With these estimates, we can see that
\eqref{*E1} and \eqref{EE2} hold.
Therefore, the Assumption $({\bf H}_V)$ is valid.

Clearly, $K\in C^\8(\mathscr D(K))$ and $K(x)\ge0$ for $x\in\mathscr D(K)$ with $|x|\le1.$
Note that for all $d\ge 2,$
\begin{equation}\label{EE5}
\frac{1}{|x|}\<x,\nn K(x)\>=-(1\vee (d-2))\frac{1}{|x|}|x|^{2-d},\quad x\in \mathscr D(K).
\end{equation}
In particular, for the case  $d=2$,
$$
\frac{1}{|x|}\<x,\nn K(x)\>=- \frac{1}{|x|} \le -\log\frac{1}{|x|}=-K(x), \quad x\in \mathscr D(K),
$$
where we used the basic inequality: $\log \alpha\le \alpha-1$ for $\alpha>0.$
Then, we deduce from \eqref{EE5} that for all $x\in\mathscr D(K)$ with $|x|\le1$,
$$
\frac{1}{|x|}\<x,\nn K(x)\>\le -K(x).
$$
Next, for each $r>0$ and $x\in\mathscr D(K)$ with $|x|\ge r$,
$$
| \nn K(x) |=(1\vee (d-2))  |x|^{1-d } \le (1\vee (d-2))  r^{1-d }<\8.
$$
Furthermore, applying \cite[Lemma 4.3]{LM} yields
\begin{align*}
&\sum_{i,j,k=1:j,k\neq i}^N \<{\bf n}( x^{(i)}-x^{(j)}),\nn  K(x^{(i)}-x^{(k)}) \>\\
&=-(1\vee (d-2))\sum_{i,j,k=1: j,k\neq i}^N\frac{1}{|x^{(i)}-x^{(k)}|^d} \<{\bf n}( x^{(i)}-x^{(j)}), x^{(i)}-x^{(k)}  \>\\
&\le -(1\vee (d-2))\sum_{i,j=1: j\neq i}^N\frac{1}{|x^{(i)}-x^{(j)}|^{d-1}}\\
&=\sum_{i,j=1:j\neq i}^N\<{\bf n}( x^{(i)}-x^{(j)}),\nn  K(x^{(i)}-x^{(j)})\>.
\end{align*}
In a word, the hypothesis $({\bf H}_K)$ is fulfilled  by the  Coulomb potential $K$, given in \eqref{EE7}.
\end{example}

\section{Strong Feller and Irreducibility}\label{section3}

In this section, we suppose that $ ( {\bf Z}_t)_{t\ge0}:=((Z_t^{(1)},\cdots,Z_t^{(N)}))_{t\ge0}$ so that, for any $1\le i\le N$,
$(Z^{(i)}_t)_{t\ge0}$ is a $d$-dimensional (rotationally invariant)  symmetric $\alpha_i$-stable L\'{e}vy process, and $ { (Z_t^{(1)})}_{t\ge0}$, $\cdots,$ $(Z_t^{(N)})_{t\ge0}$ are mutually independent. We will verify that the SDE \eqref{E1} has the strong Feller and irreducible properties, under the Lyapunov condition investigated in the preceding section.
For the strong Feller property, we will make full use of the H\"{o}rmander theorem for non-local operators (developed greatly  in \cite{Zhang, Zhang5}), invoke the truncation idea and combine  with  the  continuity
  of the Dirichlet heat kernel;
 as for the Lebesgue irreducible property, we will solve  the issue on approximate controllability  of the associated deterministic system and take advantage of the time-change idea for  SDEs driven by subordinated Brownian motions (see e.g. \cite{Zhang2, Zhang1, Zhang}).

For the sake of  simplicity of our interpretation, we stick  on  the single
particle (i.e., $N=1$) case, since the   arguments to be implemented  work essentially for general cases
on the multi-particle system
as well just accompanying  with some complicated notations. Therefore, in the following, we fix $N=1$
and write $(X_t,V_t)_{t\ge0}$ in lieu of $({\bf X}_t,{\bf V}_t)_{t\ge0}$. Moreover,
for the technical reason,
 we shall take
the driven noise $({\bf Z}_t)_{t\ge0}$ in \eqref{E1} to be a $d$-dimensional symmetric $\alpha$-stable process, which is also denoted by $(L_t)_{t\ge 0}$ in what follows.
We also assume that there exists a Lyapunov function $\mathcal V(x,v)\ge1$ for the process $(X_t,V_t)_{t\ge0}$; that is,  $\mathcal V(x,v) \to\8$
 as $H(x,v):=\frac{1}{2}\gamma|v|^2+U(x)\to \8$, and
 there exist constants $\lambda_{\mathcal V}\ge0,C_{\mathcal V}>0$ such that
\begin{equation}\label{e:Ly}
(\mathscr L\mathcal V)(x,v)\le { \lambda_{\mathcal V}} \mathcal V(x,v)+C_{\mathcal V},\quad (x,v)\in\mathscr K,
\end{equation} where $\mathscr{K}=\mathscr{D}(U)\times \R^d$ with $\mathscr{D}(U):=\{x\in \R^d: U(x)<\infty\}$.
Apparently, the Lyapunov condition \eqref{e:Ly} is fulfilled once the assumptions in Proposition \ref{pro1} or Proposition \ref{pro2} are satisfied.

In the sequel, we aim to address the issues on the strong Feller property and the irreducible property of the Markov process $(X_t,V_t)_{t\ge0}$, one-by-one.

\begin{proposition}\label{strong}Under assumptions above, the process $(X_t,V_t)_{t\ge0}$ given in  \eqref{E1} with $({\bf Z}_t)_{t\ge0}$ being a symmetric $\alpha$-stable process has the strong Feller property. \end{proposition}
\begin{proof}

For $z=(x,v)\in \mathscr K$, define
\begin{equation}\label{e:not}
b(z)=\left(
         \begin{array}{cc}
           v  \\
           -\gamma v-\nabla U(x)\\
         \end{array}
       \right),\quad \sigma=\left(
                              \begin{array}{cc}
                               {\bf0}_{d\times d} \\
                             \I_{d\times d}\\
                              \end{array}
                            \right),
                            \end{equation}
 where ${\bf 0}_{d\times d}$  and $\I_{d\times d}$     stand for the $d\times d$ zero matrix and the  $d\times d$ identity matrix, respectively.
With the notation above at hand, the SDE \eqref{E1} can be rewritten as the following compact form: for the shorthand notation $Z_t:=(X_t,V_t),$
\begin{equation}\label{e:SDE1}
\d Z_t=b(Z_t)\,\d t+\si\,\d L_t.
\end{equation}

For any $R\ge 1$, let $$\mathcal A_R=\{(x,v)\in \mathscr K: \mathcal V(x,v)< R\}.$$ In particular, $\mathcal A_R$ is an open set in $\mathscr K$. Define $U_R\in C^\infty(\R^d)$ with  bounded  derivatives such that $U_R(x)=U(x)$ for all $(x,v)\in \mathcal A_R$.
Consider the process $(Z_t^R)_{t\ge0}:=(X_t^R,V_t^R)_{t\ge0}$, which solves the SDE
\begin{equation}\label{e:SDE2-}
\d Z_t^R=b_R(Z_t^R)\,\d t+\si\,\d L_t,
\end{equation}
where $b_R$ is defined as in \eqref{e:not} with $U_R$ in place of $U$.  In particular, since $b_R(z)$ is globally Lipschitz continuous on $\R^{2d}$, the SDE \eqref{e:SDE2-} has a unique strong solution $(Z_t^R)_{t\ge0}$ which is non-explosive. Let $\si_i=\left(
                              \begin{array}{cc}
                               {\bf0}  \\
                             e_i\\
                              \end{array}
                            \right)\in\R^{2d} , i=1,2,\cdots,d$, be the $i$-th column of $\si$ defined in \eqref{e:not},  where ${\bf0}\in\R^d$ is the zero vector and $(e_i)_{1\le i\le d}$ is the standard orthogonal basis of $\R^d.$
 Then, the Lie bracket between $b_R$ and $\si_i$, denoted by  $[b_R,\sigma_i]\in\R^d$, is given by
$$[b_R,\sigma_i](z)=\nn \sigma_i\cdot b_R(z)
-\nabla b_R(z)\cdot \sigma_i =\left(
                                                          \begin{array}{cc}
                                                           - e_i \\
                                                           \gamma e_i \\
                                                          \end{array}
                                                      \right).$$
Whence, it holds that for   all $z\in \mathscr K$,
$$\mbox{Rank}\{\si_1,\cdots,\sigma_d,[b_R,\sigma_1](z),\cdots, [b_R,\sigma_d](z)\}=2d.$$   According to { \cite[Theorem 1.1 and
Theorem 1.3]{Zhang}}, the process   $(Z_t^R)_{t\ge0}$ has a transition density function  
 $\rho_{R}(t,z',z'')$ with respect to the Lebesgue measure, so that $(z',z'')\mapsto \rho_{R}(t,z',z'')$  is bounded and continuous with respect to $(z',z'')\in\mathscr K\times \mathscr K$ for any fixed $t>0$.

In the following,  let $(P_t)_{t\ge0}$ be the semigroup of the process $(Z_t)_{t\ge0}$. Fix $z=(x,v)\in \mathscr K$ and $t>0$, and  choose $R_0>0$ sufficiently large  such that $z\in \mathcal A_{R_0}$. Then, for any $R\ge R_0$ and $f\in \mathscr B_b(\mathscr K)$, we obviously have
\begin{equation}\label{strong-F}\begin{split}P_t f(z)=\Ee^{z} f(Z_t)=&\Ee^{z} \big(f(Z_t)\I_{\{\tau_{\mathcal A_R}\le t\}}\big)+ \Ee^{z}\big( f(Z_t)\I_{\{t<\tau_{\mathcal A_R}\}}\big)\\
=&\Ee^{z} \big(f(Z_t)\I_{\{\tau_{\mathcal A_R}\le t\}}\big)+\Ee^{z} \big(f(Z_t^R)\I_{\{t<\tau_{\mathcal A_R}^*\}}\big),\end{split}\end{equation} where
$$ \tau_{\mathcal A_R}:=\inf\{t>0: Z_t\notin \mathcal A_R\},\quad \tau_{\mathcal A_R}^*:=\inf\{t>0: Z_t^R\notin \mathcal A_R\}.$$ Here, in the last equality we used the fact that with the starting point $z$, the law of $Z_t^R$ coincides with  that of $Z_t$ when both associated processes do not exit the open set $\mathcal A_R$ before time $t$.

As mentioned above,    the transition density function  
 $\rho_R(t,z',z'')$ associated with $Z_t^R$ is bounded and continuous with respect to  
  $(z',z'')\in \mathscr{K}\times \mathscr{K}$ for any fixed $t>0$. Then, by following the standard approach,  we can see that the function $z\mapsto \Ee^{z} f(Z_t^R\I_{\{t<\tau_{\mathcal A_R}^*\}})$ is continuous for all $R\ge R_0$. Indeed, $\Ee^{z}\big( f(Z_t^R)\I_{\{t<\tau_{\mathcal A^*_R}\}}\big)$ corresponds to the Dirichlet semigroup of the process $(Z_t^R)_{t\ge0}$, and, thanks to arguments in \cite[Section 2.2]{CZ}, the associated transition density function (which is called the Dirichlet heat kernel in the literature) of the Dirichlet semigroup is also continuous.

Next, we turn to estimate  the term $I(t):=\Ee^{z}\big( f(Z_t)\I_{\{\tau_{\mathcal A_R}\le t\}}\big)$ for $f\in \mathscr B_b(\mathscr K)$. Note that for any $t>0$,
$$|I(t)|\le \|f\|_\infty \Pp^{z}(\tau_{\mathcal A_R}\le t).$$
According to \eqref{e:Ly} and the It\^{o} formula, it holds that for all $R\ge R_0$ and $t>0$,
$$\e^{-\lambda_{\mathcal V} t}R\, \Pp^{z}(\tau_{\mathcal A_R}\le t) \le \Ee^{z}\big(\e ^{-\lambda_{\mathcal V}(t\wedge \tau_{{\mathcal A}_R}   )}\mathcal V(Z_{t\wedge \tau_{{\mathcal A}_R}   })\big)\le \mathcal V(z)+(1-\e^{-\lambda_{\mathcal V} t} )  C_{\mathcal V}/\lambda_{\mathcal V} .$$
In particular,
\begin{equation}\label{E11}
\Pp^{z}(\tau_{\mathcal A_R}\le t)\le \frac{1}{R}\big(\mathcal V(z)+(1-\e^{-\lambda_{\mathcal V} t} )   C_{\mathcal V}/\lambda_{\mathcal V}\big)\e^{\lambda_{\mathcal V} t}.
\end{equation}
Thus, the term $I(t)$ can be neglected for $R$ large enough.

Combining with all the conclusions above, we can prove that the function $z\mapsto P_t f(z)$ is continuous, and so the process $(Z_t)_{t\ge0}$ has the strong Feller property.
\end{proof}

In order to further investigate the irreducibility of $(Z_t)_{t\ge0}$ solving \eqref{E1}, we further write $L_t=W_{S_t}$ as a form of a $d$-dimensional subordinated Brownian motion; that is, $(W_t)_{t\ge0}$  is a $d$-dimensional Brownian motion and  $(S_t)_{t\ge0}$ is an $\alpha/2$-stable subordinator, which is independent of $(W_t)_{t\ge0}. $

Below we introduce the canonical probability space corresponding to the subordinated Brownian motion $(W_{S_t})_{t\ge0}$. Let $(\mathbb W, \mathscr{B}(\mathbb W), \mu_{\mathbb W})$ be the standard Wiener space. In detail,    $\mathbb W$ is the space of all continuous functions    $\omega:\R_+\to\R^d$ with   $\omega_0={\bf0}$, which is equipped with the locally uniform convergence topology,
 and $\mu_{\mathbb W}$ is the Wiener measure, under which the coordinate process $W_t(\omega):=\omega_t$ is a standard $d$-dimensional Brownian motion.
  Let $\mathbb S$ be the space of all  increasing and  c\'{a}dl\'{a}g functions   $\ell: \R_+\to\R_+$ with $\ell_0={\bf0}$. Suppose that $\mathbb S$ is endowed with the Skorohod metric, and the probability measure $\mu_{\mathbb S}$ so that the coordinate process $S_t(\ell):=\ell_t$ is distributed with the law of the $\alpha/2$-stable subordinator. In what follows, we shall work on the probability space
\begin{equation*}
(\Omega,\mathscr F,\Pp):=(\mathbb W\times\mathbb S,\mathscr B(\mathbb W)\times\mathscr B(\mathbb S), \mu_{\mathbb W}\times\mu_{\mathbb S}).
\end{equation*}
Under this probability space, the coordinate process $L_t(\omega,\ell):=\omega_{\ell_t}$ is a symmetric $\alpha$-stable process.

\begin{proposition}\label{irre}
Assume that the Assumptions in Proposition $\ref{strong}$ hold. Then,
the process $(X_t,V_t)_{t\ge0}$ solving \eqref{E1} with $({\bf Z}_t)_{t\ge0}$ being a symmetric $\alpha$-stable process
is Lebesgue irreducible, i.e.,  for  any $z,z^*\in\mathscr K$ and $t,\vv>0$,
\begin{equation}\label{E12}
\P^z(|(X_t,V_t)-z^*|\le \vv)>0.
\end{equation}

\end{proposition}

\begin{proof}
For any $t>0$, set $Z_t:=(X_t,V_t)$.
Note that for  any $z,z^*\in\mathscr K$ and $t,\vv>0$,
\begin{equation*}
\P^z(|Z_t-z^*|\le \vv)=\int_{\mathbb S}\int_{\mathbb W}\I_{ \{|Z_t(z;\omega_{\ell_\cdot})-z^*|\le \vv\}}\,\mu_{\mathbb W}(\d\omega)\,\mu_{\mathbb S}(\d\ell).
\end{equation*}
So, in order to prove the   assertion \eqref{E12}, it suffices to verify that for any $z, z^*\in \mathscr{K}$, $\varepsilon>0$ and for $\mu_{\mathbb S}$-almost $\ell\in  \mathbb S$,
\begin{equation}\label{e:ver1}\Pp_{\mu_{\mathbb W}}^{z}(|Z_t^{\ell}-z^*|\le\vv)>0,\end{equation}  where $(Z_t^{\ell})_{t\ge0}$ is the solution to \eqref{e:SDE1} with $L_t=W_{S_t}$ replaced by $W_{\ell_t}$.

For any $\delta>0$ and $t\ge0$,   define the regular version of $\ell\in \mathbb S$ by
$$\ell^\delta_t=\frac{1}{\delta}\int_t^{t+\delta}\ell_s\,\d s=\frac{1}{\delta}\int_0^\delta\ell_{t+r}\,\d r.$$ It is clear that $\ell^\delta_t\to \ell_t$ as $\delta\to 0$; moreover, the function $t\mapsto \ell_t^\delta$ is continuous and strictly increasing. Consider the following ODE
\begin{equation}\label{e:SDE2}\d \hat Z_t=b(\hat Z_t)\,\d t+\si\,\d u_t,\end{equation} where $u\in C(\R_+;\R^{d})$, and $b$  and $\si$ are defined as in \eqref{e:not}.
 We claim that for any $t,\eta>0$, one can find $u\in C(\R_+;\R^{d})$ such that the associated solution $(\hat Z_t)_{t\ge0}$ to the ODE \eqref{e:SDE2} satisfies that $\hat Z_0=z$ and $\hat Z_t=z^*$, and
 \begin{equation}\label{E7}
\Pp_{\mu_{\mathbb W}}^{z}(A_\eta^\ell) >0,
\end{equation}
where
$$A_\eta^\ell:=\Big\{\omega\in \mathbb W: \sup_{0\le s\le t}|W_{\ell_s}(\omega)-u_s|\le \eta\Big \}.$$
 Indeed,
according to the proof of \cite[Proposition 2.5]{HM}, there is $u\in C(\R_+;\R^d)$ so that $(\hat Z_t)_{t\ge0}$ solving  \eqref{e:SDE2} satisfies that $\hat Z_0 =z $ and $\hat Z_t=z^*$. As mentioned above, for any $\delta>0$ the function $t\mapsto \ell_t^\delta$ is continuous and strictly increasing so $(W_{\ell^\delta_t})_{t\ge0}$ is still a Brownian motion with the covariation matrix  $ (\ell^\delta_t\I_{d\times d} )_{t\ge0}$. Since the support of a $d$-dimensional Brownian motion is the whole space $C(\R_+; \R^d)$, for any $t,\eta>0,$
 $$\Pp_{\mu_{\mathbb W}}^{z}\Big( \sup_{0<s\le t}|W_{\ell^\delta_s}(\omega)-u_s|\le \eta\Big)>0.$$
 On the other hand, since $\ell^\delta_t\to \ell_t$ as $\delta\to 0$, $W_{\ell^\delta_t}\to W_{\ell_t}$ a.s. for every $t>0$. Therefore, \eqref{E7} holds true by using
the fact that, for all $\ell\in \mathbb S$, the discontinuous points on $[0,t]$ for any $t\ge0$ is at most countably infinite.

 For any $r>0,$ let
$$
 E_r=\{ z\in\mathscr K: \mathcal V( z)\le r\}.
$$
Since $z\mapsto b(z)$ is locally Lipschitz on $\mathscr{K}$, there exists a constant $C_r>0$ such that
\begin{equation}\label{E8}
|b(z_1)-b(z_2)|\le C_r|z_1-z_2|,\quad z_1,z_2\in E_r.
\end{equation}
Below, let
$$R_t=R+\sup_{0\le s\le t}\mathcal V(\hat Z_s),\quad \eta_t=\frac{\vv}{ \|\si\|\e^{C_{R_t}t}}.$$

By following the argument to derive \eqref{E11}, we find that for $\mu_{\mathbb S}$-almost $\ell\in\mathbb S$,
\begin{equation*}
\lim_{r\to\8}\tau_r^\ell=\8, \quad \Pp_{\mu_{\mathbb W}}-\mbox{a.s.},
\end{equation*}
where $\tau_r^\ell:=\inf\{s>0: \mathcal V(Z_s^\ell)> r\}$. Then, for each fixed $t>0$, there exists a constant $R>0$ sufficiently large  such that
 the event $ \{t<\tau_{R_t}^\ell\} $ will occur with positive probability.

 Note from $\hat Z_t=z^*$ that
 \begin{equation*}
\begin{split}
\Pp_{\mu_{\mathbb W}}^{z}(|Z_t^{ \ell}-z^*|\le\vv)&=\Pp_{\mu_{\mathbb W}}^{z}(|Z_t^{\ell}-\hat Z_t|\le\vv  )\\
&\ge \Pp_{\mu_{\mathbb W}}^{z}\Big(\big\{\I_{\{t<\tau_{R_t}^\ell\}}|Z_t^{\ell}-\hat Z_t|\le\vv \big\}\cap A_{\eta_t}^\ell\Big ).
\end{split}
\end{equation*}
Provided that
\begin{equation}\label{E10}
A_{\eta_t}^\ell\subseteq\Big\{\omega\in \mathbb W:\I_{\{t<\tau_{R_t}^\ell(\omega_{\ell_\cdot})\}}|Z_t^{\ell}(\omega_{\ell_\cdot})-\hat Z_t|\le\vv \Big\},
\end{equation}
 we readily have
\begin{equation*}
\Pp_{\mu_{\mathbb W}}^{z}(|Z_t^{ \ell}-z^*|\le\vv)\ge \Pp_{\mu_{\mathbb W}}^{z}(A_{\eta_t}^\ell).
\end{equation*}
This, together with \eqref{E7}, yields the desired assertion.

From \eqref{E8} and the definition of $R_t$,   we deduce  that for any $  s\le t< \tau_{R_t}^\ell$,
\begin{equation*}
  \begin{split}
  \big|Z_s^\ell-\hat Z_s^\ell \big |&\le\int_0^{s } | b(Z_r^\ell)-b(\hat Z_r))|\,\d r+\|\si\| \sup_{0\le s\le t}|W_{\ell_s}-u_s|\\
  &\le C_{R_t}\int_0^{s} \big|  Z_{r}^\ell - \hat Z_{r\wedge \tau_{R_t}^\ell} \big|\,\d r+\|\si\| \sup_{0\le s\le t}|W_{\ell_s}-u_s|.
  \end{split}
\end{equation*}
Subsequently, applying Gronwall's inequality yields on the event  $  \{t< \tau_{R_t}^\ell\}$,
\begin{equation*}
\big|Z_t^\ell-\hat Z_t^\ell\big|\le\|\si\| \sup_{0\le s\le t}|W_{\ell_s}-u_s| \e^{C_{R_t} t}.
\end{equation*}
Consequently, by taking the alternative of $\eta_t$ into consideration, the inclusion \eqref{E10} is valid  when the event
 $  \{t< \tau_{R_t}^\ell\}$ takes place. Therefore, the proof is completed.
\end{proof}

\section{Proof of Theorem \ref{thm1} and General Result}\label{section4}
We first present the

\begin{proof}[Proof of Theorem $\ref{thm1}$] Suppose that the assumptions in the theorem holds.  In terms of Examples \ref{ex1} and \ref{ex2}, we know that for the potentials $U(x)$ given in (i) and (ii) of the theorem, the Lyapunov condition holds; moreover, by the proofs of Propositions \ref{pro1} and \ref{pro2}, the associated  Lyapunov function satisfies the properties of the function $V$ mentioned in the theorem.

Furthermore, according to Propositions \ref{strong} and \ref{irre},
the process  $({\bf X}_t,{\bf V}_t)_{t\ge0}$ defined by \eqref{E1} has the strong Feller and irreducible properties. Thus, by carrying out a parallel  argument of
\cite[Corollary 5.3]{HM}, where the path was laid out in that of  \cite[Lemma 2.3]{MSH},
 the locally uniform minorization condition reminiscent of Doeblin's condition is examined.
 Indeed, according to the irreducible property (see Proposition \ref{irre}), one can see that \cite[Assumption 2.1 (i)]{MSH} is satisfied. On the other hand, by the strong Feller property (see Proposition \ref{strong}) and its proof, we can verify that \cite[Assumption 2.1 (ii)]{MSH} is also fulfilled; more explicitly, as we claimed before, the second term on the right hand side of \eqref{strong-F} (i.e., $\Ee^{z} \big(f(Z_t^R)\I_{\{t<\tau_{\mathcal A_R}^*\}}\big)$ is associated with the so-called Dirichlet semigroup for the process $(Z_t^R)_{t\ge0}$, which possesses a continuous transition density function. Then, adopting the truncation argument and taking  \eqref{strong-F} into consideration again, we know that the semigroup of the original process  $({\bf X}_t,{\bf V}_t)_{t\ge0}$ also has a transition density function $\rho(t,z',z'')$ so that for every $t>0$, $(z',z'')\mapsto \rho(t,z',z'')$ is continuous on $\mathscr K\times \mathscr K$.

 With the aid of all the conclusions above, the desired assertion follows from Harris' theorem
(\cite[Theorem 1.2]{HMa}) or the proof of \cite[Theorem 2.3]{HM}. \end{proof}

\begin{remark}Since the coefficients of the SDE \eqref{E1} is locally Lipschitz continuous on $\mathscr{K}=\mathscr{D}(U)\times \R^d$ and the Lyapunov condition (see Proposition \ref{pro1} and Proposition \ref{pro2}) holds for the associated infinitesimal generator $\mathscr{L}$, we can verify that for any initial point $(x,v)\in \mathscr{K}$, there is a unique pathwise strong solution $({\bf X}_t,{\bf V}_t)_{t\ge0}$ to the SDE \eqref{E1}. Moreover, $({\bf X}_t,{\bf V}_t)\in \mathscr{K}$ almost surely for all $t>0$. For the detailed proof of the statement above, the reader can be referred to that of \cite[Proposition 2.4]{LM}. \end{remark}

 In additional  to  the main result on the specific setting (i.e., Theorem \ref{thm1}), below we present a general result on exponential ergodicity of \eqref{E1} to conclude this section. More precisely,

\begin{theorem}\label{thm2}
Assume that $({\bf H}_U)$ and $({\bf H}_\nu)$ or $({\bf H}_V)$, $({\bf H}_K)$ and $({\bf H}_\nu)$ hold true. Then, the process $({\bf X}_t,{\bf V}_t)_{t\ge0}$ solving  \eqref{E1} with $({\bf Z}_t)_{t\ge0}$ being a cylindrical symmetric stable process as that in Theorem $\ref{thm1}$
is exponentially ergodic in the sense that the process $({\bf X}_t,{\bf V}_t)_{t\ge0}$ has a  unique invariant probability measure $\mu$, and that there are a constant $\lambda$ and a positive function $C(x,v)$ so that for all $(x,v)\in \mathscr{K}:=\mathscr D(U)\times\R^{Nd}$ and $t>0$,
$$\|P_t((x,v),\cdot)-\mu\|_{V}\le C(x,v)\e^{-\lambda t},$$ where $V(x,v)\ge 1$ has the properties that  in Theorem $\ref{thm1}$.
\end{theorem}

\begin{proof}
With Propositions \ref{pro1} and \ref{pro2} at hand, the proof of Theorem \ref{thm2} can be finished exactly as that of Theorem \ref{thm1}
by keeping in mind that Propositions \ref{strong} and \ref{irre} are remain valid for the setting on multi-particle system.
Therefore, we herein   do not  go into detail.
\end{proof}

\ \

\noindent \textbf{Acknowledgements.}
The research of Jianhai Bao is   supported by NSF of China  (No.\  12071340).  The research of Rongjuan Fang is supported by NSF of China (No.\ 12201119).
The research of Jian Wang is supported by the National Key R\&D Program of China (2022YFA1000033) and  NSF of China (Nos.\ 11831014, 12071076 and 12225104).

\end{document}